\documentclass[journal]{IEEEtran}
\ifCLASSINFOpdf
  \usepackage[pdftex]{graphicx}
  \DeclareGraphicsExtensions{.pdf,.jpeg,.png}
\else
\fi
\usepackage[cmex10]{amsmath}
\usepackage{amsfonts}
\usepackage{amsthm}
\usepackage{algorithm,algorithmicx,algpseudocode}

\usepackage{caption}
\DeclareCaptionFormat{algor}{%
	\hrulefill\par\offinterlineskip\vskip1pt%
	\textbf{#1#2}#3\offinterlineskip\hrulefill}
\DeclareCaptionStyle{algori}{singlelinecheck=off,format=algor,labelsep=space}
\usepackage{array}

\ifCLASSOPTIONcompsoc
  \usepackage[caption=false,font=normalsize,labelfont=sf,textfont=sf]{subfig}
\else
  \usepackage[caption=false,font=footnotesize]{subfig}
\fi
\usepackage{mathtools}
\usepackage{enumitem}
\usepackage{xcolor}

\newcommand{\tr}{\mathrm}
\newcommand{\wt}{\widetilde}
\newcommand{\ith}{$i$-th~}

\newcommand{\RR}{\mathbb{R}}
\newcommand{\NN}{\mathbb{N}}
\newcommand{\mc}{\mathcal}

\newcommand{\tran}{^\tr{T}}
\newcommand{\eqineq}{\mathrel{\overset{\makebox[0pt]{\mbox{\normalfont\small $(\leq)$}}}{=}}}

\newcommand{\nr}{N_{i}}
\newcommand{\Cui}{C_{i}^{\tr{u}}}
\newcommand{\Cci}{C_{i}^{\tr{c}}}
\newcommand{\Cpi}{C_{i}^{\tr{p}}}
\newcommand{\Qppi}{Q_{\tr{pp},i}}
\newcommand{\Quui}{Q_{\tr{uu},i}}
\newcommand{\Qpui}{Q_{\tr{pu},i}}
\newcommand{\params}{\left[\Phi_i\tran \; \Theta_i\tran \right]}

\newcommand{\tLs}{\lambda_\tr{L}^\star}
\newcommand{\tUs}{\lambda_\tr{U}^\star}
\newcommand{\tL}{\lambda_\tr{L}}
\newcommand{\tU}{\lambda_\tr{U}}
\newcommand{\til}{\lambda_i^\tr{l}}
\newcommand{\tiu}{\lambda_i^\tr{u}}

\newcommand{\etae}{\eta_{\tr{e},i}}

\newtheorem{theorem}{Theorem}

\newtheorem{lemma}{Lemma}
\newtheorem{remark}{Remark}
\newtheorem{assumption}{Assumption}
\newtheorem{proposition}{Proposition}

\DeclareMathOperator{\med}{median}

\DeclareMathOperator{\sign}{sign}

\DeclareMathOperator{\diag}{diag}

\begin{document}
%
\title{
Parametric Optimization Based MPC for Systems of Systems with Affine Coordination Constraints 
}
%
%
%

\author{Branimir~Novoselnik,~
		Vedrana~Spudi\'{c},~
        and~Mato~Baoti\'{c}
\thanks{B. Novoselnik and M. Baoti\'{c} are with University of Zagreb, Faculty of Electrical Engineering and Computing, Unska 3, HR-10000 Zagreb, Croatia, e-mail: \texttt{\{branimir.novoselnik, mato.baotic\}@fer.hr}.}
\thanks{V. Spudi\'{c} is with ABB Corporate Research Center, Segelhofstrasse 1k, CH-5404 Baden-D\"{a}ttwil, Switzerland, e-mail: \texttt{vedrana.spudic@ch.abb.com}}
\thanks{This research has been supported by the European Commission's FP7-ICT project DYMASOS (contract no. 611281), by the Croatian Science Foundation (contract no. I-3473-2014), and by the European Regional Development Fund under the grant KK.01.1.1.01.0009 (DATACROSS).}
}

\maketitle

\begin{abstract}
A large-scale complex system comprising many, often spatially distributed, dynamical subsystems with partial autonomy and complex interactions are called system of systems. This paper describes an efficient algorithm for model predictive control of a class of system of systems for which the overall objective function is the sum of convex quadratic cost functions of (locally) constrained linear subsystems that are coupled through a set of (global) linear constraints on the subsystems coordination parameters. 

The proposed control algorithm is based on parametrization and splitting of the underlying optimization problem into one global coordination problem and a set of local optimization problems pertaining to individual subsystems. The local optimization problems are solved off-line, via parametric optimization, while the coordination problem is solved on-line. The properties of the local parametric solutions are utilized to solve the coordination problem very efficiently. In particular, it is shown that, for a fixed number of coupling constraints, the coordination problem can be solved with a linear-time algorithm in a finite number of iterations if all subsystems have one-dimensional coordination parameters.
\end{abstract}


%
\IEEEpeerreviewmaketitle

\section{Introduction}

\IEEEPARstart{M}{any} socio--technical systems consist of a large number of partly autonomous subsystems with local self‐management that are coupled by physical interactions via streams of energy (electricity, steam) or material (water, gas, intermediates), forming so-called System of Systems (SoS) \cite{Engell2015,Kopetz2015,cpsos}. Examples of such systems are the electrical grid, buildings and building complexes, petrochemical and chemical production plants, water distribution systems, and gas networks. There is inherently a conflict between the local optimization performed by individual subsystems and the goals of the overarching SoS or of society as a whole. For example, in the electrical distribution grid, the local consumers want maximum comfort, guaranteed power supply and a low cost of electricity, the grid operator wants to maintain grid stability and high revenues by cheap generation, low transmission losses and high sales prices, and society wants to minimize the carbon footprint. 

Constituent subsystems and their interactions, resources and goals clearly need to be coordinated if common (global) desired outcomes are to be achieved. Moreover, the dynamic interaction of locally managed subsystems gives rise to complex dynamic behavior of the overall SoS. Hence, the lack of proper coordination in SoS can easily lead to large-scale disruptions, e.g. black-outs in the electrical grid. 

The coordination of SoS can often be formulated as a Model Predictive Control (MPC) problem. The main idea of the discrete-time MPC is to forecast system behavior, as a function of control inputs, by using a dynamic model of the system that starts from a known (measured or estimated) initial state. For a chosen cost function on a (finite) prediction horizon and prescribed state/input constraints, the MPC computes the optimal sequence of control inputs. Only the first element of the optimal sequence is applied to the system and the entire procedure is repeated at the next sampling instant, starting from a new initial state. For more details on the theory of MPC the reader is referred to \cite{MPCbook}. 

For practical applicability of the MPC it is vitally important to be able to solve the underlying optimization problem at every time instant, i.e., within the sampling time. There are basically two approaches to ensure this: i) the (off-line) computation of the optimal control law, or ii) the use of fast/tailored (on-line) constrained optimization algorithms. 

The first approach -- the so-called explicit MPC -- employs multi-parametric programming to compute the solution to the underlying optimization problem as an explicit function of the initial state. Such a function is precomputed off-line for all possible values of the initial state. Hence, the on-line computation reduces to a simple function evaluation \cite{Bemporad2002}. The main drawback of explicit MPC is that the complexity of the off-line solution can grow exponentially with the number of constraints in the control problem formulation. In practice, this means that explicit MPC is applicable only for systems with small number of states, control inputs and/or constraints.

The second approach consists of solving the underlying MPC optimization problem (in most cases a linear or a quadratic program) on-line at every sampling instant. Unfortunately, the general purpose solvers are of limited use when one has to solve computationally demanding optimization problem, as is the case in optimal coordination of complex SoS. The optimization problems arising in MPC algorithms often have specific structure that can be exploited to tailor the optimization algorithm and obtain the solution efficiently. Several methods are proposed in the literature for efficient interior point methods tailored to convex multistage problems arising in MPC applications, c.f. \cite{Wang2010,Mattingley2010,Domahidi2012,FORCESPro}. However, a well-known drawback of these methods is their limited warm start capability. A tailored active set strategy for the fast solution of quadratic programs (QPs) arising in MPC was proposed in \cite{Ferreau2008,Ferreau2014}. Although this method fully exploits the knowledge of similarity between the solutions of subsequently solved QPs, it does not benefit from the problem sparsity as much as interior point methods do. A dual Newton strategy that builds on ideas of interior point methods but still features warm start capabilities of active set methods is proposed in \cite{Frasch2015} for solving strictly convex QPs in the MPC setting. 

Although the aforementioned, specifically tailored MPC algorithms are very efficient, they all assume that the control problem can be solved in a centralized fashion. Unfortunately, this is not true in general, especially in the case of complex SoS. In many instances the central coordinator does not have full information about the local variables and constraints, e.g. because of privacy issues, different ownerships, conflicting economic goals, or management structures. In such cases fully centralized coordination based on the MPC approach is inapplicable \cite{cpsos}. The state-of-the-art approaches found in literature usually deal with this problem by decentralization of the computation -- they decompose the original, large optimal control problem into a number of smaller and more tractable subproblems that can be solved in parallel (i.e., independently). The overview of different approaches for distributed MPC based on various distributed optimization methods can be found in survey papers, e.g. \cite{Camponogara2002,Christofides2013,Negenborn2014}. Different techniques can be found in the literature for distributed optimization but most of these methods heavily exploit the concepts of convexity and duality, i.e., the original problem is often decoupled by introducing some dual variables to relax coupling constraints. The problem is then to find the optimal dual variables by maximizing the dual function. Typically, an iterative (global) aggregation/(local) optimization procedure is employed to find the globally optimal solution. The downside of these distributed optimization approaches is that they usually require a great number of calculation/communication iterations to converge to the solution \cite{Christofides2013,Boyd2011}.  

This paper describes a different, centralized MPC approach that combines both the on-line and off-line computation for coordinated control of SoS. The two main goals of the proposed method are: i) to mimic favorable properties of classical distributed optimization methods so that it is applicable in an SoS framework (e.g. the protection of data privacy of individual subsystems) and (ii) to have an efficient and scalable on-line computation algorithm that can easily be applied to large SoS. A considerable amount of the computational effort is transferred to the off-line procedure thus allowing for more efficient on-line computation. The overall control problem is decomposed into smaller, decoupled local problems related to individual subsystems and one coordination problem that describes the coupling of the subsystems. The contribution of each local subsystem to the overall cost function and its optimal control actions are determined off-line, as the solution to a multi-parametric Quadratic Program (mp-QP). This is also motivated by real-life examples of SoS, e.g. a power system where the grid coordinator needs to coordinate different generators in order to meet the demand for electrical power in the grid, while the generators declare their local behavior to the coordinator with the cost-of-generation functions. The globally optimal coordination parameters are determined on-line by solving the coordination problem based on the off-line solutions. There is no need for iterative exchange of information between the central coordinator and local subsystems to ensure the convergence to a globally optimal solution -- a limited amount of information (parts of local solution obtained off-line) is sent to the coordinator by each subsystem only once per sampling time. Thus, the central coordinator does not need to know everything about local subsystems in order to find the globally optimal solution. The optimal coordination problem has a specific structure that rests on the theoretic properties of multi-parametric solutions of local problems. The problem structure enables design of a linear-time on-line coordination algorithm that finds the solution in a finite number of iterations and therefore, by utilizing efficient primal decomposition of the optimization problem, allows the application of MPC to an SoS with a large number of subsystems. 

In \cite{Trnka2016} a similar idea of solving a distributed MPC with a combined explicit-iterative approach is described. Trnka et al. target a class of strictly convex quadratic optimization problems with linear constraints, where complicating (global) variables are coupled by equality constraints. Parametric solutions of local subproblems are used in each iteration to build the gradient and Hessian of a dual function and the global coordination problem (the search for dual function maximum) is solved by a damped Newton method. Although we focus on a less general formulation of the optimization problem where a dimension of global variable at each subsystem is equal to 1 (whereas in \cite{Trnka2016} an arbitrary dimension is allowed), in our case the global coordination problem exhibits a special structure that we exploit to develop an efficient linear-time algorithm that scales well with the overall size of the SoS.

This paper significantly extends on the previous work \cite{Spudic2013} in which a special case of a single coupling constraint was considered. An algorithm presented here can handle more coupling constraints. Furthermore, an in-depth description of the on-line coordination algorithm that runs in linear time is made, together with benchmark results from a numerical case study.

The main contributions of the paper are:

\begin{itemize}
	\item It is shown that, under the considered problem formulation, the global coordination problem can be reformulated as a (multiply constrained) continuous quadratic knapsack problem.
	
	\item We describe an efficient linear-time algorithm, denoted as the hyperplane searching (HPS) algorithm, that solves the global coordination problem in a finite number of iterations thus allowing for a superior scalability to large-scale SoS. The algorithm generalizes the breakpoint searching (BPS) algorithm to the case of multiple coupling constraints. 
	
	\item The developed control algorithm is tested on a power systems numerical case study that shows favorable computational properties of the proposed approach. Numerical results not only empirically confirm theoretical properties of the on-line algorithm (i.e. its linear-time complexity) but also highlight  its practical usefulness - drastic reductions in on-line computation time are achievable (up to two orders of magnitude smaller) compared to the classical centralized MPC.
\end{itemize}

The rest of this paper is organized as follows. In Section~\ref{sec:problem_setup} the control problem -- MPC for SoS -- is defined, while Section~\ref{sec:solution_method} describes the proposed solution method. In Section~\ref{sec:efficient_algorithm} a linear-time algorithm for the coordination problem is described. The efficiency of the proposed approach is tested on a numerical case study of a microgrid system in Section~\ref{sec:illustrative_example}, followed by the concluding remarks in Section~\ref{sec:conclusion}.

\section{Problem Setup}
\label{sec:problem_setup}

We consider coordinated MPC of an SoS comprising $M$ coupled, locally controllable subsystems described by linear time-invariant (LTI) dynamics
\begin{align}
\label{eq:sub_unit_model}
x_i(t+1) = A_i x_i(t) + B_i u_i(t),
\end{align}
where $t\in\mathbb{Z}$ denotes discrete time, $i\in\left\{1,\ldots,M\right\}$ is the subsystem index, $x_i(t)\in\mathbb{R}^{n_{\text{x},i}}$ and $u_i(t)\in\mathbb{R}^{n_{\text{u},i}}$ denote state and input of \ith subsystem at time $t$, respectively, while $A_i$ and $B_i$ are constant matrices of appropriate dimensions.


It is assumed that the control objective for every LTI subsystem \eqref{eq:sub_unit_model}, when considered in isolation, can be expressed as a local MPC problem with some finite prediction horizon, quadratic cost and linear constraints on the subsystems' states and inputs. Since every linear MPC problem can be formulated as a multi-parametric Quadratic Program (mp-QP), see \cite{Bemporad2002} for details, in the rest of the paper it is assumed, without loss of generality, that the local control problem has the form of the following mp-QP
\begin{subequations}
	\label{eq:local_problem}
	\begin{alignat}{2}
	J_i^\star(\Phi_i,\Theta_i) = \; & \underset{U_i}{\text{min}}
	& & \; J_i(\Phi_i,\Theta_i,{U}_i),\\
	& \;\text{s.t.} 
	\label{eq:local_constraints}
	& & \; \Cui {U}_i \leq \Cci + \Cpi \params\tran,
	\end{alignat}
with a convex quadratic cost function
\begin{align}
\label{eq:local_cost}
\begin{split}
J_i(\Phi_i,\Theta_i,{U}_i) &:= \params\Qppi\params\tran
+ {U}_i\tran\Quui{U}_i \\&+ \params\Qpui{U}_i,
\end{split}
\end{align}
\end{subequations}
where $U_i\in\mathbb{R}^{n_{\text{U},i}}$ is the vector of local optimization variables (e.g. control inputs on the prediction horizon), $\Phi_i\in\mathbb{R}^{n_{\Phi,i}}$ and $\Theta_i\in\mathbb{R}^{n_{\Theta,i}}$ are the parameters, $\Qppi=\Qppi\tran \succeq 0$, $\Quui = \Quui\tran \succ 0$ and $\Qpui$ are suitably sized cost matrices, while $\Cui$, $\Cci$ and $\Cpi$ are suitably sized constraint matrices. The properties of the optimizer ${U}_{i}^{\star}(\Phi_i,\Theta_i)$ and the value function $J_{i}^{\star}(\Phi_i,\Theta_i)$ are summarized in the following theorem.
\begin{theorem}[see \cite{Bemporad2002}]\label{th:mpqp_solutions}
Consider the mp-QP \eqref{eq:local_problem}. 
The set of feasible parameters
\begin{align}
\mc{P}_i:=\left\{ \params\tran \;\vert\; \exists {U}_i : \Cui {U}_i \leq \Cci + \Cpi \params\tran \right\}
\end{align}
is a polyhedral set, the value function $J_i^\star : \mc{P}_i \rightarrow \RR$ is a convex and continuous piecewise quadratic function on polyhedra (PPWQ), and the optimizer ${U}_i^\star : \mc{P}_i \rightarrow \RR^{n_{\emph{U},i}}$ is a continuous piecewise affine function on polyhedra (PPWA).
\end{theorem}

Note that the mp-QP \eqref{eq:local_problem} has two types of parameters: local parameters $\Phi_i$ and coordination parameters $\Theta_i$. The vector of local parameters $\Phi_i$ includes all data that will be locally available (obtained from measurements, estimations or predictions from historical data; e.g., initial state of the \ith subsystem) at the time instant when MPC problem has to be solved. The coordination parameters $\Theta_i$ describe the contribution of the \ith subsystem to the coupling constraints.


The coupling between constituent subsystems in the SoS is described by $m$ linear (in)equalities of the following form
\begin{align}
\label{eq:coupling_constraints}
\textstyle\sum\limits_{i=1}^{M} a_{i,j}\tran \Theta_i \eqineq b_j, \quad j=1,\ldots,m,
\end{align}
where 
$a_{i,j}\in\mathbb{R}^{n_{\Theta,i}}$ and $b_j\in\mathbb{R}$ are constant parameters. These kinds of constraints typically arise in resource allocation problems \cite{Stojanovski2015} and distributed production problems \cite{Kundur1994}. Note that inequality constraints can be treated exactly the same as equality constraints, e.g. by introduction of slack variables. Therefore, in the rest of the paper, without loss of generality, only equality coupling constraints in \eqref{eq:coupling_constraints} are considered. 

\subsection*{Coordination problem}

The coordinated SoS aims to minimize the cumulative cost of all subsystems while satisfying their local constraints and global coupling constraints, i.e., one needs to solve the following (centralized) coordination problem
\begin{subequations}
	\label{eq:coordination_problem}
	\begin{alignat}{2}
	\label{eq:coordination_cost}
	& \underset{\substack{{U}_1,\ldots,{U}_M\\\Theta_1,\ldots,\Theta_M}}{\text{min}}
	& & \; \textstyle\sum\limits_{i=1}^{M} J_i(\Phi_i,\Theta_i,{U}_i),\\
	&\hspace{0.46cm} \text{s.t.} 
	\label{eq:coordination_local_constraints}
	& & \; \Cui {U}_i \leq \Cci + \Cpi \params\tran, \; i=1,\ldots,M,\\
	\label{eq:coordination_coupling_constraints}
	& & & \; \textstyle\sum\limits_{i=1}^{M} a_{i,j}\tran \Theta_i = b_j, \quad j=1,\ldots,m.
	\end{alignat}
\end{subequations}
Note that $\Theta_i$ is treated as a parameter in the local control problem \eqref{eq:local_problem}, while in the coordination problem \eqref{eq:coordination_problem} it is one of the optimization variables.


The overall control algorithm -- MPC for SoS -- runs in a receding horizon fashion. The coordination problem \eqref{eq:coordination_problem} is solved at every time sample $t$ for fixed (measured or estimated) values of parameters $\Phi_1(t),\ldots,\Phi_M(t)$. Only the first element $u_{0,i}^\star$ of the optimizer $U_i^\star$ is applied to \ith subsystem using the control law $u_i(t) = u_{0,i}^\star$, $i=1,\ldots,M$, and the entire procedure is repeated at the next sampling instant. Although \eqref{eq:coordination_problem} is a quadratic program -- the global cost function \eqref{eq:coordination_cost} is a sum of convex quadratic local cost functions \eqref{eq:local_cost} -- that can be solved by a centralized controller, we aim to solve it more efficiently and in a manner that allows for protection of subsystems data.

{
	\begin{remark}
		Note that local control problems \eqref{eq:local_problem} are considered to be heterogeneous in terms of prediction horizon length, number of local constraints, number of subsystem states and control inputs. 
		
		For simplicity, however, it is assumed that all local subsystems have the same sampling time. Note that one could also handle the case with different sampling times -- provided they are all integer multiples of coordinator's sampling time $T_s$. In such case subsystems with sampling times larger than $T_s$ would also have to provide a prediction of their behavior between their respective samples. 
	\end{remark}
}

\begin{remark}
	\label{rem:electrical_grid_example}
The dynamic optimal dispatch problem in electrical grid operation and control \cite{Kundur1994} can be formulated as \eqref{eq:coordination_problem}. In that case $M$ subsystems are dispatchable generators, where $J_i(\Phi_i,\Theta_i,U_i)$ is the cost of power generation, $\Phi_i$ is initial state ($\Phi_i(t)=x_{i}(t)$), $U_i$ is the vector of control inputs on the prediction horizon, and $\Theta_i$ is the desired generated electrical power in steady state. The goal is to minimize the total cost of generation \eqref{eq:coordination_cost}, while satisfying local constraints pertaining to individual generators \eqref{eq:coordination_local_constraints} as well as power balance coupling constraints \eqref{eq:coordination_coupling_constraints} induced by the grid itself.
\end{remark}

\subsection*{Dimension of the coordination parameter}

If the coordination parameter $\Theta_i$ is allowed to have an arbitrary dimension, practically any control problem based on model predictive control of linear(ized) SoS with quadratic cost would be covered by this problem setup. When $n_{\Theta,i} > 1$ an iterative approach outlined in \cite{Spudic2011}  can be used to solve the coordination problem. Naturally, larger coordination parameters render more complex coordination problems and as the parameter size increases, this approach is expected to be less efficient. In this paper, however, we are focused on the case when $n_{\Theta,i}=1$ because in this case a very efficient on-line algorithm that runs in linear time and retrieves the globally optimal solution of \eqref{eq:coordination_problem} in a finite number of iterations can be constructed. The requirement that the dimension of the coordination parameter is equal to one is somewhat restrictive in terms of applicability to a general class of SoS, e.g. dynamically coupled systems. However, it allows very efficient solution methods. In the example of electrical grid from Remark~\ref{rem:electrical_grid_example}, the subsystems were coupled by flows of only one main product $\Theta_i$ -- electrical power. One could, in principle, introduce multiple secondary products to the coupling constraints \eqref{eq:coupling_constraints} if the secondary products were affine functions of the main product, i.e. the coupling constraints \eqref{eq:coupling_constraints} could still be expressed in terms of $\Theta_i$. One such example is the network of combined heat and power generators that can produce both the electrical power (main product) and heat (secondary product) from some primary energy source (e.g. natural gas). This example is further described in Sec. \ref{sec:illustrative_example}.

\section{Solution method}
\label{sec:solution_method}

Since \eqref{eq:coordination_problem} is a convex optimization problem 
it can be restated as follows
\begin{equation}
\label{eq:coordination_problem2}
	\begin{array}{cl}
	\underset{\Theta_1,\ldots,\Theta_M}{\text{min}} &\hspace{-0.2cm} \textstyle\sum\limits_{i=1}^{M}
	\left\{
	\begin{array}{cl}
	\hspace{-0.15cm}\underset{{U}_i}{\text{min}} &\hspace{-0.2cm} J_i(\Phi_i,\Theta_i,{U}_i)\\
	\hspace{-0.15cm}\text{s.t.} &\hspace{-0.2cm} \Cui {U}_i \leq \Cci + \Cpi \params\tran
	\end{array}
	\right.
	\\[3ex]
	\text{s.t.} &\hspace{-0.2cm} \textstyle\sum\limits_{i=1}^{M} a_{i,j}\tran \Theta_i = b_j, \quad j=1,\ldots,m.
	\end{array}
\end{equation}

From \eqref{eq:coordination_problem2} it is clear that the coordination of SoS, i.e., solution to \eqref{eq:coordination_problem}, can be achieved with a hierarchically structured controller, illustrated in Fig.~\ref{fig:on_line_steps}, with two levels of control:
\begin{itemize}
\item A set of \emph{local controllers} assigned to each subsystem at the bottom layer of the control hierarchy. Each localized computational unit has full knowledge of the local control problem \eqref{eq:local_problem} data/variables, but it has no knowledge about other subsystems.
\item The \emph{central coordinator}, which sits at the top of the controller hierarchy, is a centralized computational unit responsible for solving the coordination problem. It can communicate with all local controllers, but it does not have direct access to any of the subsystems. The central coordinator has full knowledge of the coupling constraints, but the knowledge of the local control problems is limited to the parts which the local controllers are willing to share.
\end{itemize}


With the hierarchical structure in Fig.~\ref{fig:on_line_steps} in mind, the method proposed in this paper for MPC of SoS combines off-line and on-line computation phases. Firstly, in the off-line phase, the local controllers solve their local control problems \eqref{eq:local_problem} parametrically, to obtain the optimizer ${U}_{i}^{\star}(\Phi_i,\Theta_i)$ and the value function $J_{i}^{\star}(\Phi_i,\Theta_i)$ as closed-form functions, which can be readily done, e.g. by solving the corresponding mp-QP with the MPT toolbox \cite{MPT3}.



In the on-line phase, at every time sample $t$, the computation is done in the following steps:
\begin{enumerate}
\item \textit{Local evaluation}: For all $i\in\{1,\ldots,M\}$, obtain measurement/estimation of local parameters $\hat{\Phi}_i$ and evaluate the local control problem solution for it
\begin{subequations}
\label{eq:evaluated_local_solutions}
\begin{align}
\wt{U}_{i}(\Theta_i) &:= {U}_{i}^{\star}(\hat{\Phi}_i,\Theta_i),\;\; \wt{U}_i : \mc{I}_i \rightarrow \RR^{n_{\text{U},i}},\\
\label{eq:evaluated_value_function}
\wt{J}_{i}(\Theta_i) &:= J_{i}^{\star}(\hat{\Phi}_i,\Theta_i),\;\; \wt{J}_{i} : \mc{I}_i \rightarrow  \RR,
\end{align}
\end{subequations}
where $\mc{I}_i$ is an interval in $\RR$, obtained by slicing $\mc{P}_i$ with $\Phi_i = \hat{\Phi}_i$, i.e., $\mc{I}_i=\left\{ \Theta_i \;\vert\; [\hat{\Phi}_i\tran,\Theta_i]\tran\in \mc{P}_i\right\}$.

\item \textit{Solve the coordination problem}: Given the evaluated local value functions \eqref{eq:evaluated_local_solutions}, solve the coordination problem \eqref{eq:coordination_problem} reformulated in the following equivalent form
\vspace{-0.2cm}
\begin{subequations}
\label{eq:evaluated_coordination_problem_formulation}
\begin{alignat}{2}
& \underset{\Theta_1,\ldots,\Theta_M}{\text{min}}
& & \; \textstyle\sum\limits_{i=1}^{M} \wt{J}_{i}(\Theta_i), \\
&\hspace{0.46cm} \text{s.t.} 
& & \;\Theta_i \in \mc{I}_i,\;\; i=1,\ldots,M,\\
& & & \; \textstyle\sum\limits_{i=1}^{M} a_{i,j} \Theta_i = b_j, \quad j=1,\ldots,m,
\end{alignat}
\end{subequations}
to obtain the optimal value of coordination parameters $\Theta_{i}^{\star}$, $i=1,\ldots,M$.

\item \textit{Local evaluation}: For all $i\in\{1,\ldots,M\}$, evaluate the optimal control input
\begin{equation}
{U}_{i}^{\star} = \wt{U}_{i}(\Theta_{i}^{\star}).
\end{equation}
and apply $u_i(t) = u_{0,i}^\star$ as a control input to the \ith subsystem. Note that $u_{0,i}^\star$ is either contained in ${U}_i^\star$ or can be straightforwardly computed from $\Theta_i^\star$ and ${U}_i^\star$.
\end{enumerate}
The overview of the on-line steps are illustrated in Fig.~\ref{fig:on_line_steps}. The local evaluations are performed in parallel on the local controllers, while the coordination problem is solved by the central coordinator. Notice the flow of information: the local controllers send the description of the value function $\wt{J}_i:\mc{I}_i \rightarrow \RR$ (up to $3 n_{\tr{r},i}+1$ numbers, where $n_{\tr{r},i}$ is the number of critical regions in the off-line solution of the \ith subsystem) to the central coordinator, while the central coordinator returns the value of the optimal coordination parameter $\Theta_{i}^{\star}$ (only one number per subsystem). We point out that $3 n_{\tr{r},i}+1$ is the size of $\wt{J}_i$ in the worst-case scenario. In practice, the number of intervals in $\wt{J}_i$ is expected to be much lower than the total number of regions $n_{\tr{r},i}$ in the local parametric solution since $\wt{J}_i$ is just a slice of the entire $J_i^\star$ with fixed local parameters $\hat{\Phi}_i$. This exchange of information is performed only once per time sample. Hence, there is no need for iterative communication between the central coordinator and local controllers in order to reach the global optimum, like with most of classical distributed optimization techniques which usually need a large number of such iterations (see e.g. \cite{Kozma2015}). Please note that our total communication effort (the amount of data transferred) can generally be as good (or bad) as other distributed optimization methods, if the number of intervals in the description of $\wt{J}_i$ is excessively high. However, we point out that this was not the case in our case study where only a small amount of data was transferred in each iteration. Any further analysis of the communication effort is out of scope of this paper.

\begin{figure}[!t]
\centering
\includegraphics[width=0.45\textwidth]{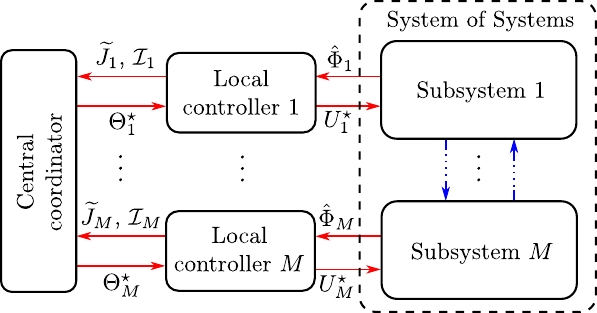}
\caption{The hierarchical controller structure for coordination of SoS. Red arrows indicate communication links and blue arrows indicate the coupling between individual subsystems.}
\label{fig:on_line_steps}\vspace{-0.5cm}
\end{figure}

The solution obtained in the on-line steps outlined above is globally optimal and it is exactly the same as the solution that would be obtained by solving \eqref{eq:coordination_problem}.

Moreover, notice that the local controllers share only limited information with the central coordinator. Potentially confidential data like subsystem matrices, constraints and cost functions are masked behind evaluated functions \eqref{eq:evaluated_value_function} which are scalar piecewise quadratic functions defined on intervals $\mc{I}_i$.

The local computation boils down to evaluation of PPWA and PPWQ functions, a relatively simple task that can be performed efficiently and thus introduces a negligible computational overhead. A real computational burden lies in solving \eqref{eq:evaluated_coordination_problem_formulation} so an efficient algorithm for \eqref{eq:evaluated_coordination_problem_formulation} is needed. 

\section{Efficient Coordination Algorithm}
\label{sec:efficient_algorithm}


\subsection{Reformulation of the coordination problem}
\label{subsec:reformulation}

Note that $\wt{J}_i$ defined in \eqref{eq:evaluated_value_function} is a scalar convex piecewise quadratic function defined on a closed interval $\mc{I}_i\subset\RR$, i.e.,
\begin{subequations}
		\label{eq:ppqw_fja}
	\begin{align}
	&\wt{J}_i(\Theta_i) = \textstyle\frac{1}{2} h_{i,r} \Theta_i^2 + f_{i,r}\Theta_i + g_{i,r} \ \mathrm{if}~ \Theta_i \in \left[ I_{i,r-1}, I_{i,r} \right],\\
	&\mc{I}_i = \bigcup_{r=1}^{\nr} \left\{ \Theta_i \;\vert\;  I_{i,r-1} \leq \Theta_i \leq I_{i,r} \right\} = \left[ I_{i,0}, I_{i,\nr} \right],
	\end{align}
\end{subequations}
where $\nr$ is the number of subintervals in partition of $\mc{I}_i$, scalars $h_{i,r}$, $f_{i,r}$ and $g_{i,r}$ are parameters of the quadratic function in $r$-th subinterval, while $I_{i,r}\in\mathbb{R}$ are the endpoints of those subintervals, $r=1,\ldots,\nr$, $i=1,\ldots,M$. Consequently, the cost function of \eqref{eq:evaluated_coordination_problem_formulation} is a PPWQ function defined on a Cartesian product of intervals, i.e. on hyperrectangles in $\RR^M$.

The optimization problem \eqref{eq:evaluated_coordination_problem_formulation} is historically known as a monotropic piecewise quadratic program (PQP) (cf. \cite{Sun1986,Sun1991,Sun1992}). PQP is found in many practical problems, especially those involving network structures, variable costs, stochastic factors, soft constraints but it can also arise as a subproblem in solving more complex mathematical programs \cite{Sun1991}. In \cite{Rockafellar1994} an iterative simplex-based algorithm that solves PQP directly has been developed. In our work we pursue a different indirect approach.

As is shown in Appendix\ref{appendix}, the optimization problem \eqref{eq:evaluated_coordination_problem_formulation} can be reformulated as a convex separable quadratic program with box constraints and coupling equality constraints that has the following general form
\begin{subequations}
\label{eq:separable_QP}
\begin{alignat}{2}
& \underset{x_1,\ldots,x_n}{\text{min}}
& & \quad \textstyle\sum\limits_{i=1}^{n} \frac{1}{2} d_i x_i^2 - a_i x_i, \\
& \hspace{0.35cm}\text{s.t.} 
& & \quad l_i \leq x_i \leq u_i, \; i=1,\ldots,n,\\
\label{eq:separable_QP_coupling}
& & & \quad B x = c, 
\end{alignat}
\end{subequations}
where $x=\left[x_1,\ldots,x_n\right]\tran\in\RR^n$ is the optimization variable, while
$d=\left[d_1,\ldots,d_n\right]\tran\in\RR^n$, $d_i\geq0$, $a=\left[a_1,\ldots,a_n\right]\tran\in\RR^n$, $l=\left[l_1,\ldots,l_n\right]\tran\in\RR^n$, $u=\left[u_1,\ldots,u_n\right]\tran\in\RR^n$, $l_i\leq u_i$, $B=[B_1,\ldots,B_n]\in\RR^{m\times n}$, and $c\in\RR^{m}$ are (known) parameters, with $m\in\NN$ being the number of coupling constraints in \eqref{eq:separable_QP_coupling}. The total number of variables $n$ in \eqref{eq:separable_QP} is equal to $\sum_{i=1}^{M} N_i$ (see Appendix\ref{appendix} for details).


\begin{remark}
In literature problem \eqref{eq:separable_QP} is called a \emph{continuous quadratic knapsack problem} (CQKP) when $m=1$ (single coupling constraints), and a \emph{multiply constrained continuous quadratic knapsack problem} (MCQKP) otherwise \cite{knapsackBook}.
\end{remark}

\subsection{Single coupling constraint}
\label{subsec:single_coupling_constraint}

There are two very efficient methods for solving \eqref{eq:separable_QP} when $m=1$ known in literature: the breakpoint searching (BPS) algorithm \cite{BRUCKER1984,Kiwiel2008} and the variable fixing (VF) algorithm \cite{Kiwiel2008vf}. The BPS algorithm solves \eqref{eq:separable_QP} in $\mc{O}(\log{n})$ iterations in time $\mc{O}(n)$, but it heavily depends on an efficient implementation of median searching algorithm. The VF algorithm has simpler implementation since it uses only elementary algebraic operations. On average the VF algorithm requires $\mc{O}(n)$ iterations and the worst-case performance is $\mc{O}(n^2)$. In practice both algorithms have similar average run times \cite{Kiwiel2008vf}. The focus in the rest of the paper is on the BPS algorithm since it can be nicely generalized to handle the case $m>1$. 

For $m=1$ the problem \eqref{eq:separable_QP} can be written compactly as
\begin{subequations}
\label{eq:knapsack}
\begin{alignat}{2}
& \underset{x}{\min}
& & \quad \textstyle\frac{1}{2} x\tran D x - a\tran x\\
& \hspace{0.15cm}\mathrm{s.t.} 
& & \quad l \leq x \leq u,\\
\label{eq:knapsack_coupling}
& & & \quad b\tran x = c, 
\end{alignat}
\end{subequations}
where $D=\diag\left( d \right)$, $b=\left[b_1,\ldots,b_n\right]\tran\in\RR^n$, $c\in\RR$. 
Let $x^*\in\RR^n$ denote a minimizer to \eqref{eq:knapsack}. 

Next we give a detailed description of the known BPS algorithm \cite{BRUCKER1984,Kiwiel2008} for solving \eqref{eq:knapsack}, with an extension (which, to the best of our knowledge, is not available in the literature) to handle the cases when some $d_i=0$. For simplicity, in the rest of this subsection the following assumption holds.
\begin{assumption}
\label{as:positiveb}
Problem \eqref{eq:knapsack} is feasible and $b > 0$.
\end{assumption}
Note that Assumption~\ref{as:positiveb} is non-restrictive. If some $b_i=0$ then $x_i$ does not contribute to the coupling constraint \eqref{eq:knapsack_coupling} and $x^*_i$ can be easily computed (e.g., $x^*_i=\mathrm{median}\{l_i, a_i/d_i, u_i\}$ if $d_i>0$) thus reducing dimension of the problem to $n-1$. If some $b_i < 0$, one can use variable substitution $\tilde{x}_i=-x_i$ (and solve the new problem for which $\tilde{b}_i>0$). Finally, for $b>0$, the problem \eqref{eq:knapsack} is feasible if and only if $b\tran l\leq c \leq b\tran u$.

The Lagrangian relaxation of \eqref{eq:knapsack} is
\begin{subequations}
\label{eq:knapsack_lagrange}
\begin{alignat}{2}
\phi(\lambda):=\quad & \underset{x}{\min}
& & \quad \textstyle\frac{1}{2} x\tran D x - a\tran x + \lambda(b\tran x-c) \\
& \hspace{0.15cm}\mathrm{s.t.} 
& & \quad l \leq x \leq u, 
\end{alignat}
\end{subequations}
where $\lambda\in\RR$ is a multiplier for the equality constraint \eqref{eq:knapsack_coupling}. It is well known that $\phi(\lambda)$ is concave and that maximizing $\phi(\lambda)$ is equivalent to solving \eqref{eq:knapsack} \cite{BoydBook}. For any given $\lambda$ it is easy to evaluate $\phi(\lambda)$ because it has a separable structure
\begin{align}\label{eq:phi_lam_bps}
\phi(\lambda) = \phi_1(\lambda)+\ldots+\phi_n(\lambda) - \lambda c,
\end{align}
\begin{align}
\label{eq:defphi_i}
\phi_i(\lambda) := \underset{x_i}{\text{min}}\left\{\textstyle \frac{1}{2} d_i x_i^2 + (\lambda b_i - a_i)x_i : l_i \leq x_i \leq u_i \right\}.
\end{align}
Let $x_i(\lambda)$, $i=1,\ldots,n$, denote a minimizer to \eqref{eq:defphi_i}, then solving \eqref{eq:knapsack} amounts to finding a multiplier $\lambda^*$ in the optimal dual set, \cite{Kiwiel2008},
\begin{align}
\label{eq:Lambdastar}
\Lambda^\star := \left\{ \lambda: b\tran x(\lambda)=c \right\}=\left[\tLs,\tUs\right]\subset\RR.
\end{align}
The BPS algorithm does this by updating $\lambda$ to a new value if $b\tran x(\lambda)\neq c$, while simultaneously improving the (over)estimate of $\Lambda^*$. To help describe the BPS algorithm it is useful to define a scalar function
\begin{align}
\label{eq:glambda}
g(\lambda):= b\tran x(\lambda),
\end{align}
which can be readily evaluated for any $\lambda$ since explicit expressions for $x_i(\lambda)$ are available:
\begin{align}
\label{eq:xilambda}
x_i(\lambda) = \begin{cases}
l_i & \mathrm{if}~ d_i=0 \wedge \lambda b_i > a_i,\\
\in\left[ l_i, u_i \right] & \mathrm{if}~ d_i=0 \wedge \lambda b_i = a_i,\\
u_i & \mathrm{if}~ d_i=0 \wedge \lambda b_i < a_i,\\
\mathrm{median}\{l_i, \frac{a_i-\lambda b_i}{d_i}, u_i\} & \mathrm{if}~ d_i>0.
\end{cases}
\end{align}
The situation when $(d_i=0) \wedge (\lambda b_i = a_i)$ needs to be treated carefully since in that case $x_i(\lambda)$ is not uniquely defined by \eqref{eq:xilambda}. To resolve this ambiguity one can find all such indices 
\begin{align}
&\mc{I}:=\left\{ i: d_i = 0,\; \lambda b_i = a_i \right\},
\end{align}
then compute 
\begin{align}
\label{eq:LUands}
\bar{L} = \textstyle\sum_{i\in \mc{I}} b_i l_i, \ \bar{U} = \textstyle\sum_{i\in \mc{I}} b_i u_i, \
s=\sum_{i\notin \mc{I}} b_i x_i(\lambda),
\end{align}
and evaluate $x_i(\lambda)$, $\forall i\in\mc{I}$, depending on which of the following conditions is met:
\begin{itemize}
\item[i)] $\bar{L}\leq c-s \leq \bar{U}$. The optimum has been found, i.e., $\lambda\in\Lambda^*$. With a straightforward inspection, using \eqref{eq:LUands}, one can confirm that the choice
\begin{align}
\label{eq:xilambdai}
\textstyle x_i(\lambda) = l_i + \frac{c-s-\bar{L}}{\bar{U}-\bar{L}} (u_i - l_i), \; \forall i \in \mc{I},
\end{align}
satisfies the coupling constraint $g(\lambda)=c$.
\item[ii)] $\bar{U} < c-s$. The optimum has not been found, $g(\lambda) < c$.
\begin{align}
\label{eq:xilambdaii}
x_i(\lambda) = u_i, \; \forall i \in \mc{I}.
\end{align}
\item[iii)] $\bar{L} > c-s$. The optimum has not been found, $g(\lambda) > c$. \begin{align}
\label{eq:xilambdaiii}
x_i(\lambda) = l_i, \; \forall i \in \mc{I}.
\end{align}
\end{itemize}

From \eqref{eq:glambda}--\eqref{eq:xilambdaiii} it follows that $g(\lambda)$ is a piecewise affine, non-increasing function of $\lambda$ that is either continuous (if all $d_i > 0$) or has a finite number of discontinuities (if some $d_i = 0$), see Fig. \ref{fig:function_g} for illustration. Clearly, $g(\lambda) > c$ if and only if $\lambda < \tLs$ and $g(\lambda) < c$ if and only if $\lambda > \tUs$. In general, $g(\lambda)$ has $2n$ breakpoints
\begin{align}
\label{eq:breakpoints}
\til := \frac{a_i - l_i d_i}{b_i}, \quad \tiu := \frac{a_i - u_i d_i}{b_i}, \quad i=1,\ldots,n.
\end{align}
at which it changes slope or makes a jump. Note that $\tiu \leq \til$ since $l_i \leq u_i$ and $b_i > 0$. If $d_i = 0$ then $\til = \tiu = a_i/b_i$.

\begin{figure*}[htbp]
\centering
\subfloat[Illustration of $x_i(\lambda)$, $b_i > 0$ and $d_i > 0$.]
{\includegraphics[width=2in]{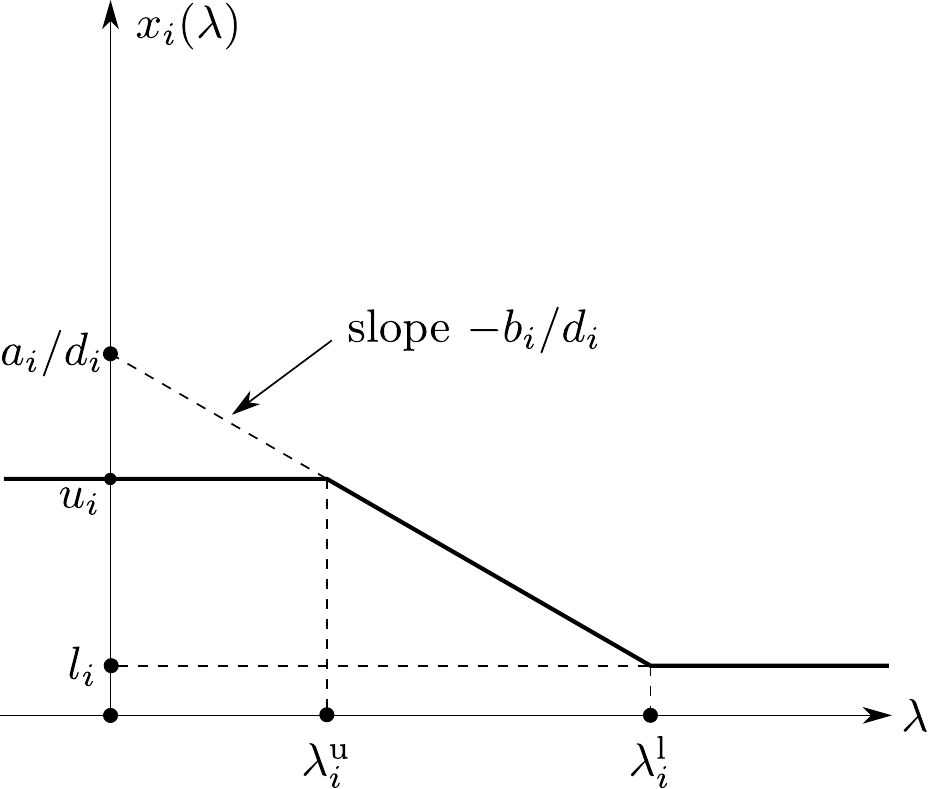}%
\label{fig:xodt}}
\hfil
\subfloat[Illustration of $x_i(\lambda)$, $b_i > 0$ and $d_i = 0$.]{\includegraphics[width=2in]{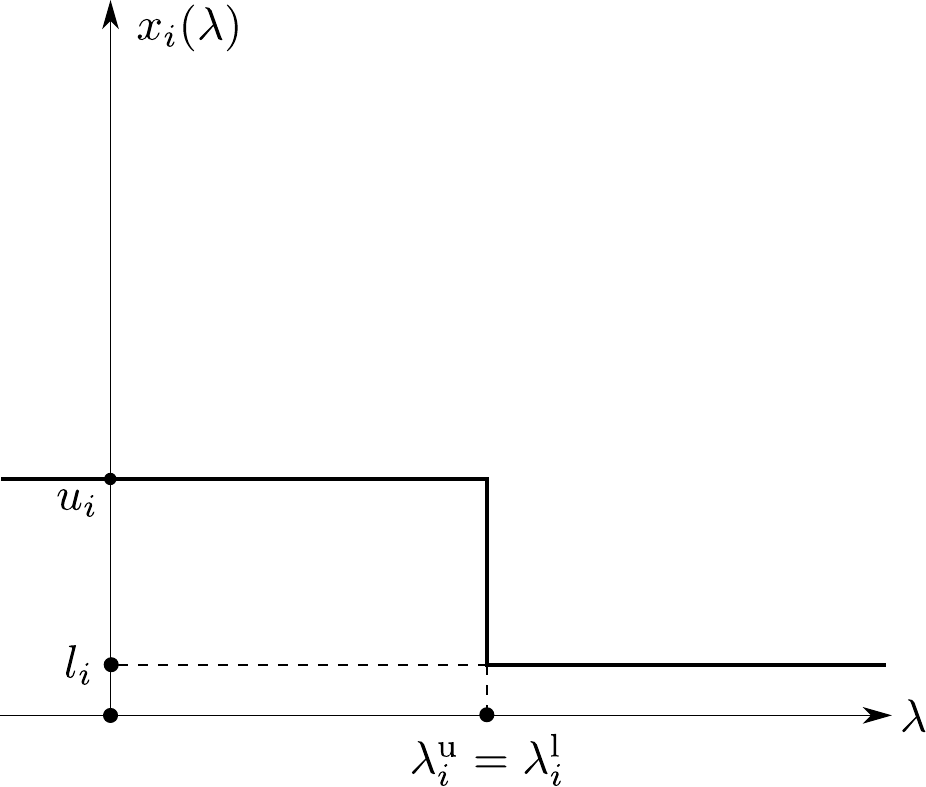}%
	\label{fig:xodt2}}
\hfil
\subfloat[Illustration of $g(\lambda)$ when some $d_i = 0$.]{\includegraphics[width=2in]{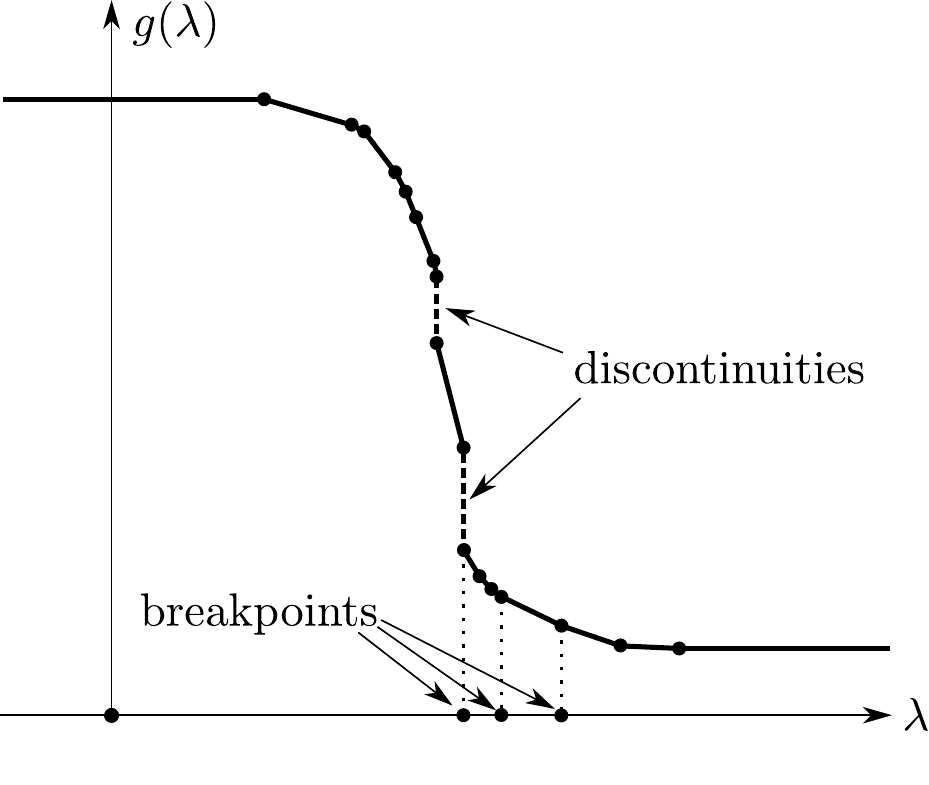}%
	\label{fig:funkcijag_disc}}
\caption{Illustration of $x_i(\lambda)$ and $g(\lambda) = b\tran x(\lambda)$.}
\label{fig:function_g}\vspace{-0.5cm}
\end{figure*}


\begin{center}
	\captionof{algorithm}{Breakpoint searching algorithm, cf. \cite{Kiwiel2008}}
	\label{alg:bps_algorithm}
	\begin{algorithmic}[1]
		\Require Parameters $l$, $u$, $d$, $a$, $b$, $c$ of \eqref{eq:knapsack}
		\Ensure $x^\star = \left[x_1^\star,\ldots,x_n^\star\right]\tran$
		\Procedure{BPS}{$l$, $u$, $d$, $a$, $b$, $c$}
		\State \parbox[t]{.8\linewidth}{$\mathcal{L} \leftarrow \{\lambda_1^{\mathrm{l}}, \lambda_1^{\mathrm{u}}, \ldots, 
			\lambda_n^{\mathrm{l}}, \lambda_n^{\mathrm{u}} \}$, where $\lambda_i^{\mathrm{l}}$ and $\lambda_i^{\mathrm{u}}$, with $i=1,\ldots,n$, are computed as in \eqref{eq:breakpoints}
		} 
		\State $\tL \leftarrow \min \mathcal{L}$, $\tU \leftarrow \max \mathcal{L}$
		\While{\textbf{true}}
		\State $\hat{\lambda} \leftarrow \med{\left\{\mathcal{L}\right\}}$
		\If{$g(\hat{\lambda}) = c$}
		\State $\lambda^\star \leftarrow \hat{\lambda}$ \label{alg:bps_algorithm_median}, \textbf{break}
		\ElsIf{$g(\hat{\lambda}) > c$}
		\State $\tL \leftarrow \hat{\lambda}$, $\mathcal{L} \leftarrow \left\{ \lambda\in \mathcal{L} ~|~ \lambda > \hat{\lambda} \right\}$
		\ElsIf{$g(\hat{\lambda}) < c$}
		\State $\tU \leftarrow \hat{\lambda}$, $\mathcal{L} \leftarrow \{ \lambda\in \mathcal{L} ~|~ \lambda < \hat{\lambda} \}$
		\EndIf
		\If{$\mathcal{L} = \emptyset$}
		\State $\lambda^\star \leftarrow \tL - \left[ g(\tL) - c \right] \frac{\tU - \tL}{g(\tU) - g(\tL)}$, \textbf{break} \label{alg:bps_algorithm_interpolation}
		\EndIf
		\EndWhile
		\State \Return $x^\star \leftarrow x(\lambda^\star)$
		\EndProcedure
	\end{algorithmic}	
	\vskip1pt\offinterlineskip\hrulefill
\end{center}

A simple implementation of the BPS algorithm is listed in Alg.~\ref{alg:bps_algorithm}. It generates successive nondecreasing underestimates $\tL$ of $\tLs$ and nonincreasing overestimates $\tU$ of $\tUs$ by evaluating $g(\lambda)$ at trial breakpoints in $\left[\tL,\tU\right]$ until $\tL$ and $\tU$ become two consecutive breakpoints; then $g(\lambda)$ is linear on $\left[\tL,\tU\right]$ and $\lambda^\star$ is found by simple interpolation.

Note that all steps in Alg.~\ref{alg:bps_algorithm} can be executed in $\mc{O}(\left\vert{\mathcal{L}}\right\vert)$ time, if proper care is taken to use previous calculations (in particular when evaluating $g(\hat{\lambda})$, see \cite{Kiwiel2008}). Since initially $\left\vert{\mathcal{L}}\right\vert=2n$, and with every pass through the while loop $\left\vert{\mathcal{L}}\right\vert$ is reduced by half, the following proposition holds.

\begin{proposition}[see \cite{Kiwiel2008} for details]
	The BPS algorithm in Alg.~\ref{alg:bps_algorithm} has a linear worst-case time complexity $\mc{O}(n)$.
\end{proposition}

Consequently, it has been shown that the coordination problem \eqref{eq:evaluated_coordination_problem_formulation}, in the case of one coupling constraint, can be solved with an efficient, linear-time algorithm.

\subsection{Multiple coupling constraints}
\label{subsec:multiple_coupling_constraints}

The Lagrangian relaxation of the general MCQKP \eqref{eq:separable_QP} is:
\begin{subequations}
\label{eq:multiple_knapsack_lagrange}
\begin{alignat}{2}
\phi(\lambda):=\quad & \underset{x}{\text{min}}
& & \quad \textstyle \frac{1}{2} x\tran D x - a\tran x + \lambda\tran(B x-c), \\
& \hspace{0.15cm}\text{s.t.} 
& & \quad l \leq x \leq u, 
\end{alignat}
\end{subequations}
where $\lambda\in\RR^{m}$, $B=[B_1, \ldots, B_n]\in\RR^{m\times n}$, $c\in\RR^{m}$, with $m\geq 2$ denoting the number of coupling equality constraints. 
As before, $\phi(\lambda)$ is a separable, concave function
\begin{align}
\label{eq:lambdasep}
\phi(\lambda) = \phi_0(\lambda)+\phi_1(\lambda)+\ldots+\phi_n(\lambda),
\end{align}
where
\begin{align}
\label{eq:phi_i}
\phi_i(\lambda) &:= \underset{x_i \in [l_i, u_i]}{\text{min}}\ \textstyle \frac{1}{2} d_i x_i^2 + (\lambda\tran B_i - a_i)x_i, \ i=1,\ldots,n\\
\phi_0(\lambda) &= - \lambda\tran c.
\end{align}
Solving \eqref{eq:separable_QP} is the same as finding
\begin{align}
\label{eq:lambdastar}
\lambda^\star \in \arg \underset{\lambda}{\mathrm{max}} \ \phi(\lambda).
\end{align}
For simplicity of exposition, in the rest of this section it is assumed that $d_i > 0$, $i=1,\ldots, n$. Therefore, the minimizer in \eqref{eq:phi_i} takes a simplified form
\begin{align}
\label{eq:x_i_od_lambda}
x_i(\lambda) = \mathrm{median}\{l_i, (a_i-\lambda\tran B_i)/d_i, u_i\}.
\end{align}
The optimization problem \eqref{eq:multiple_knapsack_lagrange} can be interpreted as an mp-QP if $\lambda$ is treated as a parameter. From Theorem \ref{th:mpqp_solutions}, $\phi(\lambda)$ is a piecewise quadratic concave function over polyhedral partition of $\lambda$ space (i.e. $\RR^m$) imposed by $2n$ hyperplanes:
\begin{equation}
\label{eq:hyperplanes}
\mathcal{H}_i =\{\lambda ~|~  h_{i,0} + h_{i}\tran \lambda = 0 \}, \ i=1,\ldots,2n,
\end{equation}
where $h_{i,0}\in\RR$ and $h_{i}\in\RR^m$, $i=1,\ldots,n$, are defined as:
\begin{equation}
\label{eq:hypparam}
h_i=h_{i+n}:=B_i,\ h_{i,0}:=d_i l_i - a_i,\ h_{i+n,0}:=d_i u_i - a_i.
\end{equation}
An example of $\phi(\lambda)$, for $n=3$ and $m=2$, is illustrated in Fig. \ref{fig:phi_od_lambda}, with different colors denoting different regions. 

\begin{figure}[!t]
	\centering
	\includegraphics[width=0.45\textwidth]{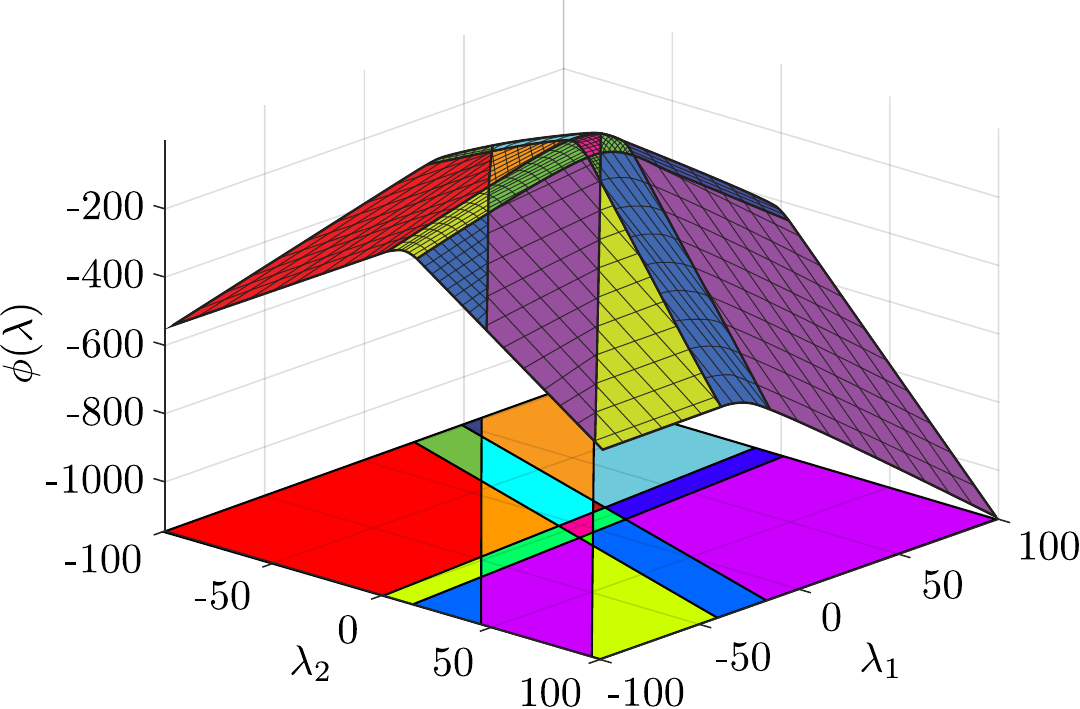}
	\caption{An example of a function $\phi(\lambda)$ for a MCQKP \eqref{eq:separable_QP} with $n=3,\;m=2$.}
	\label{fig:phi_od_lambda}\vspace{-0.5cm}
\end{figure}

A key element in the following computations is the availability of an \emph{oracle} -- an algorithm that takes as inputs parameters $p_0$ and $p\neq 0$ of a hyperplane $\mathcal{P}\subset\RR^m$,
\begin{equation}
\label{eq:generichyp}
\mathcal{P} = \{\lambda\in\RR^m ~|~ p_0 + p\tran\lambda=0 \}
\end{equation}
and parameters of function $\phi:\RR^m \rightarrow \RR$ in \eqref{eq:multiple_knapsack_lagrange}, and then returns the \emph{sign of hyperplane}, i.e. information about the relative position of 
$\lambda^\star$ with respect to that hyperplane:
\begin{equation}
\label{eq:hypsign}
\sign(\mathcal{P}) := \mathrm{sign}(p_0 + p\tran\lambda^\star) \in \{-1, 0, 1\}.
\end{equation}
One possible implementation of an oracle is given in Alg.~\ref{alg:oracle}.

Note that if $\sign(\mathcal{H}_i)$ and $\sign(\mathcal{H}_{n+i})$ are known, for some $i\in\{1,\ldots,n\}$, then the exact expression for $x_i(\lambda)$ in \eqref{eq:x_i_od_lambda} is also known because
\begin{subequations}
\label{eq:explicitxi}
\begin{align}
&\sign(\mathcal{H}_i)\in\{0, 1\} \Rightarrow \sign(\mathcal{H}_{n+i})=1, \ x_i(\lambda) = l_i,\\
&\sign(\mathcal{H}_{n+i})\in\{-1, 0\} \Rightarrow \sign(\mathcal{H}_{i})=-1, \ x_i(\lambda) = u_i,\\
&\sign(\mathcal{H}_i)=-1 \wedge \sign(\mathcal{H}_{n+i})=1 \Rightarrow \ x_i(\lambda) = \textstyle\frac{a_i-\lambda\tran B_i}{d_i},
\end{align}
\end{subequations}
and, consequently, one can obtain an explicit (in general quadratic) form of $\phi_i(\lambda)$ in \eqref{eq:phi_i}. 
It is convenient to aggregate all known explicit expressions for $\phi(\lambda)$ in $\phi_0(\lambda)$
\begin{equation}
\label{eq:explicitphi0}
\begin{array}{rcl}
\phi_0(\lambda)&:=&\frac{1}{2}\lambda\tran H_0\lambda+F_0\tran\lambda+G_0= \\
&=& -c\tran\lambda + \sum_{i\in\mathcal{E}} \left.\phi_i(\lambda)\right|_{x(\lambda) \ \mathrm{from} \ \eqref{eq:explicitxi}},
\end{array}
\end{equation}
where $\mathcal{E}$ is the set of indices of $\phi_i(\lambda)$ with known explicit expressions
\begin{equation}
\label{eq:discovered}
\mathcal{E}:=\{i\in\{1,\ldots,n\} ~|~ (i\notin \mathcal{U}) \wedge (i+n\notin \mathcal{U})\}.
\end{equation}
and $\mathcal{U}\subseteq\{1,\ldots,2n\}$ is the set of indices of hyperplanes whose signs are not yet determined. Note that \eqref{eq:lambdasep} now becomes
\begin{equation}
\label{eq:phi4alg}
\phi(\lambda) = \phi_0(\lambda) + \sum_{j\in\{1,\ldots,n\} \backslash \mathcal{E}} \phi_j(\lambda),
\end{equation}
with the implicit $\phi_j(\lambda)$ defined by \eqref{eq:phi_i}.

The proposed algorithm for finding $\lambda^\star$ is called the hyperplane searching (HPS) algorithm, since it is essentially a generalization of the BPS algorithm to multiple coupling constraints -- with the notion of breakpoints \eqref{eq:breakpoints} being replaced with the notion of hyperplanes \eqref{eq:hyperplanes}. A pseudo-code of the HPS algorithm is listed in Alg. \ref{alg:hps_algorithm}.

\begin{center}
	\captionof{algorithm}{Hyperplane searching algorithm}
	\label{alg:hps_algorithm}
	\begin{algorithmic}[1]
		\Require Index set of hyperplanes with unknown signs, $\mathcal{U}$, parameters $l$, $u$, $d$, $a$, $B$ of \eqref{eq:separable_QP}, $\phi_0(\cdot)$ in \eqref{eq:explicitphi0}
		\Ensure $x^\star = \left[x_1^\star,\ldots,x_n^\star\right]\tran$
		\Procedure{HPS}{$\mathcal{U}$, $l$, $u$, $d$, $a$, $B$, $\phi_0$}
		\State $[m,n] \leftarrow \mathrm{dim}(B)$
		\State Define hyperplanes $\mathcal{H}_{\mathcal{U}}$ as in \eqref{eq:hyperplanes}--\eqref{eq:hypparam}
		
		\While{$\mathcal{U}\neq \emptyset$}
		\State $[\mathcal{I}, \sign(\mathcal{H}_{\mathcal{I}})] \leftarrow \mathrm{MDS}(\mathcal{H}_{\mathcal{U}},\mathcal{U},l,u,d,a,B,\phi_0)$
		\For{$i\in \mathcal{I}$} 
		\State $\mathcal{U} \leftarrow \mathcal{U} \backslash i$
		\State \textbf{if} $i>n$ \textbf{then} $i\leftarrow i-n$ \textbf{end if}
		\If{$(i \notin \mathcal{U}) \wedge (i+n\notin \mathcal{U})$}
		\State \parbox[t]{.70\linewidth}{Compute explicit expression for ${\phi}_{i}(\cdot)$ via \eqref{eq:explicitxi}, and 
			update ${\phi}_0(\cdot) \leftarrow {\phi}_0(\cdot)+{\phi}_{i}(\cdot)$}
		\EndIf
		\EndFor
		\EndWhile
		\State $\lambda^\star \leftarrow \arg \underset{\lambda}{\mathrm{max}} \ {\phi}_0(\lambda) = -H_0^{-1}F_0$ 
		\State \Return $x^\star \leftarrow x(\lambda^\star)$
		\EndProcedure
	\end{algorithmic}	
	\vskip1pt\offinterlineskip\hrulefill
\end{center}

In each iteration of the HPS algorithm the sign of a \emph{fixed proportion} of $|\mathcal{U}|$ hyperplanes in $\RR^m$ (with currently unknown signs) is determined using the multidimensional search (MDS) algorithm, which was first introduced in \cite{Megiddo1984} and later improved in \cite{Dyer1986,Clarkson1986}. The MDS algorithm takes as inputs the set $\mathcal{H}_{\mathcal{U}}$ of hyperplanes in $\RR^m$, the index set $\mathcal{U}$ of those hyperplanes, parameters of the original problem \eqref{eq:separable_QP}, and ${\phi}_0(\lambda)$ -- the aggregate of the currently discovered explicit expressions for $\phi$ (starting with ${\phi}_0(\lambda)=-c\tran \lambda$, see \eqref{eq:explicitphi0}--\eqref{eq:phi4alg}). The MDS algorithm returns the index set $\mathcal{I}$ of hyperplanes whose signs have been deduced and the values of those signs, $\sign(\mathcal{H}_{\mathcal{I}})$. Once the signs of all hyperplanes are known the explicit expression for $\phi(\lambda)$ is available and the computation of $\lambda^\star$ in \eqref{eq:lambdastar} becomes simple unconstrained maximization of a concave quadratic function ${\phi}_0(\lambda)$. In the end, one calculates the optimizer to \eqref{eq:separable_QP} as $x^\star = x(\lambda^\star)$. 

The MDS algorithm runs in a recursive manner, starting with the initial problem at the level $m$, and then explores, in a depth-first approach, a binary tree down to the (bottom) level $1$. When moving towards lower levels, the MDS uses appropriate variable transformations to create $2$ sets of hyperplanes (of, roughly, half the size of the parent problem), whose effective dimensions (in the new coordinates) are reduced by one. At the bottom level, with one query to the oracle (passing the description of the hyperplane in the initial, level $m$ coordinates), the MDS algorithm can deduce the signs of roughly half of hyperplanes on that level. Due to the properties of transformations that are used to reduce dimensions of hyperplanes, it is possible to determine the signs of a portion of hyperplanes at level $\ell+1$, from the signs of hyperplanes at level $\ell$. The intricate details of these transformations and other implementation details of the general MDS algorithm can be found in \cite{Megiddo1984,Dyer1986,Clarkson1986}. 

As mentioned before, the MDS algorithm computes the signs of a fixed proportion of $|\mathcal{U}|$ hyperplanes, while making, at the bottom level of its exploration strategy, a finite number of queries to the oracle (e.g., Alg.~\ref{alg:oracle}). A key feature of the MDS algorithm is that the number of those queries depends (exponentially) only on the number of coupling constraints, $m$, and does not depend on $|\mathcal{U}|$, cf. \cite{Megiddo1984,Dyer1986}.

\begin{center}
	\captionof{algorithm}{Oracle algorithm -- sign of hyperplane $\mathcal{P}\subset\RR^m$}
	\label{alg:oracle}
	\begin{algorithmic}[1]
		\Require Parameters of $\mathcal{P}$ as in \eqref{eq:generichyp}, index set of hyperplanes with unknown signs, $\mathcal{U}$,
		parameters $l$, $u$, $d$, $a$, $B$ of \eqref{eq:separable_QP}, $\phi_0(\cdot)$ in \eqref{eq:explicitphi0}
		\Ensure $\sigma = \sign(\mathcal{P})$
		\Procedure{ORACLE}{$p_0$, $p$, $\mathcal{U}$, $l$, $u$, $d$, $a$, $B$, ${\phi}_0$}
		\State $[m,n] \leftarrow \mathrm{dim}(B)$
		\If{$m>1$}
		\State \parbox[t]{.8\linewidth}{
			Compute parameters $\bar{l}$, $\bar{u}$, $\bar{d}$, $\bar{a}$, $\bar{B}$, and $\bar{\phi}_0(\cdot)$,
			of restriction of $\phi$ in \eqref{eq:phi4alg} on $\mathcal{P}$, i.e. $\bar{\phi}:\RR^{m-1}\rightarrow \RR$,
			\begin{equation*} 
			\bar{\phi}(\bar{\lambda})\leftarrow \phi\left(\left[\begin{array}{c}
			\bar{\lambda} \\ 
			-\frac{p_0}{p_m}-\sum_{j=1}^{m-1} \frac{p_j}{p_m} \bar{\lambda}_j
			\end{array}\right]\right)
			\end{equation*}
		}
		\State Find the optimal solution for the restriction
		\begin{eqnarray*}
			\bar{\lambda}^\star &\leftarrow & \mathrm{HPS}(\mathcal{U}, \bar{l}, \bar{u}, \bar{d}, \bar{a}, \bar{B}, \bar{\phi}_0),\\
			\lambda^{\star}_{\mathcal{P}}&\leftarrow &\left[\begin{array}{c}
				\bar{\lambda}^{\star} \\ 
				-\frac{p_0}{p_m}-\sum_{j=1}^{m-1} \frac{p_j}{p_m} \bar{\lambda}^{\star}_j
			\end{array}\right],
		\end{eqnarray*}
		\Else
		\State $\lambda^{\star}_{\mathcal{P}} \leftarrow -\frac{p_0}{p_1}$
		\EndIf
		\State \parbox[t]{.85\linewidth}{Compute subgradient of $\phi$ at 
			$\lambda^{\star}_{\mathcal{P}}\in\mathcal{P}$,
			and deduce the value of $\sign(\mathcal{P})$ from the values of 
			derivative of $\phi$ at $\lambda^{\star}_{\mathcal{P}}$ in the direction $p$ and $-p$.}
		\State \Return $\sigma \leftarrow \sign(\mathcal{P})$
		\EndProcedure
	\end{algorithmic}	
	\vskip1pt\offinterlineskip\hrulefill
\end{center}

In Alg.~\ref{alg:oracle}, the oracle deduces the value of $\sign(\mathcal{P})$ by using subgradient of $\phi$ at $\lambda^{\star}_{\mathcal{P}}\in\mathcal{P}=\{\lambda ~|~ p_0+p\tran\lambda = 0\}$. The idea is as follows. Let $\mathcal{H}_{\mathcal{U}}$ be the set of all hyperplanes with unknown signs, formed as in \eqref{eq:hyperplanes}--\eqref{eq:hypparam}, with $l$, $u$, $d$, $a$, being the parameters passed to the oracle. 
Clearly, for arbitrarily small $\epsilon>0$, one can readily compute the explicit expressions for $\phi$ in \eqref{eq:phi4alg} at $\lambda=\lambda_{\mathcal{P}}^\star+\epsilon p$,
\begin{equation}
\textstyle\phi_{+}(\lambda) = \frac{1}{2}\lambda\tran H_{+} \lambda + F_{+}\tran \lambda + G_{+},
\end{equation}
and at $\lambda=\lambda_{\mathcal{P}}^\star-\epsilon p$,
\begin{equation}
	\textstyle\phi_{-}(\lambda) = \frac{1}{2}\lambda\tran H_{-} \lambda + F_{-}\tran \lambda + G_{-},
\end{equation}
where matrices $H_{+}$, $H_{-}$, $F_{+}$, $F_{-}$, $G_{+}$, and $G_{-}$, are obtained similarly to \eqref{eq:explicitxi}--\eqref{eq:explicitphi0}, by checking on which side of each $\mathcal{H}_i$ lies the corresponding $\lambda$. Computations can be sped up by noticing that the expressions for $\phi_{+}$ and $\phi_{-}$ differ only for the set of hyperplanes that intersect $\mathcal{P}$ at the point $\lambda^{\star}_{\mathcal{P}}$. After computation of gradients:
\begin{subequations}
	\begin{align}
		g_{+} &:= \textstyle\frac{\partial}{\partial\lambda}\phi_{+}(\lambda)\lvert_{\lambda=\lambda_\mathcal{P}^\star} = H_{+} \lambda_\mathcal{P}^\star + F_{+},\\
		g_{-} &:= \textstyle\frac{\partial}{\partial\lambda}\phi_{-}(\lambda)\lvert_{\lambda=\lambda_\mathcal{P}^\star} = H_{-} \lambda_\mathcal{P}^\star + F_{-},
	\end{align}
\end{subequations}
one can determine the position of the global optimizer $\lambda^\star$ relative to $\mathcal{P}$ as follows:
\begin{equation}
\sign(\mathcal{P})=\left\{
\begin{array}{lcl}
-1 & \mathrm{if} & (p\tran g_{+}<0) \wedge (p\tran g_{-}<0),\\
\phantom{-}0 & \mathrm{if} & (p\tran g_{+}\leq 0) \wedge (p\tran g_{-}\geq 0),\\
\phantom{-}1 & \mathrm{if} & (p\tran g_{+}>0) \wedge (p\tran g_{-}>0).
\end{array}
\right.
\end{equation}
Note that other cases, e.g., $(p\tran g_{+}>0) \wedge (p\tran g_{-}<0)$, cannot happen due to concavity of $\phi(\lambda)$. Figure \ref{fig:subgrad} illustrates the situation when $\sign(\mathcal{P})=+1$.

\begin{figure}[!t]
	\centering
	\includegraphics[width=0.33\textwidth]{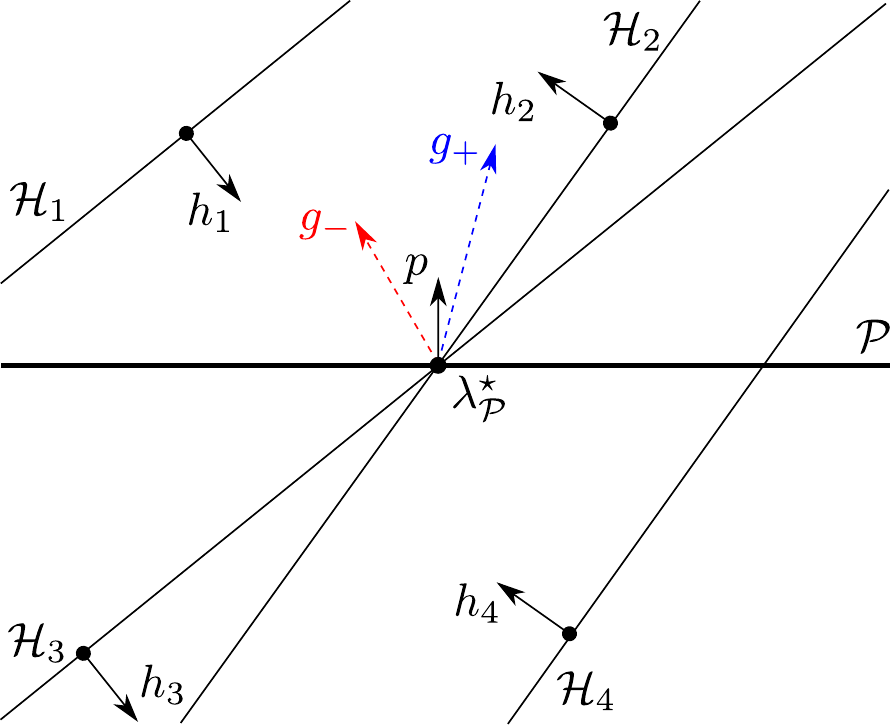}
	\caption{Illustration of the subgradient method for determining the position of $\lambda^\star$ relative to a hyperplane $\mathcal{P}$. A situation where $\sign(\mathcal{P})=+1$ is shown.}
	\label{fig:subgrad}\vspace{-0.5cm}
\end{figure}

In summary, we have the following proposition.
\begin{proposition}\label{prop:hps_algorithm_is_linear}
	For a fixed number of coupling constraints, $m$, the HPS algorithm (Alg.~\ref{alg:hps_algorithm}) solves problem \eqref{eq:separable_QP} with complexity $\mc{O}(n)$.
\end{proposition}
\begin{proof}
Note that all steps in Alg.~\ref{alg:oracle} can be executed in $\mc{O}(\left\vert{\mathcal{U}}\right\vert)$ time, if one can solve the HPS algorithm at dimension $m-1$ in $\mc{O}(\left\vert{\mathcal{U}}\right\vert)$. Clearly, for $m=1$ the HPS algorithm has the same complexity as the BPS algorithm, $\mc{O}(n)$, since the MDS algorithm in this case involves one median calculation and one call to the oracle that invokes no recursions. The proof then goes by induction on the dimension $m$. For arbitrary $m$ the computational effort in each iteration of the HPS algorithm is $\mc{O}(|\mathcal{U}|)$, where $|\mathcal{U}|$ is the number of remaining hyperplanes with unknown signs at the beginning of that iteration, plus the effort of a constant number of queries to the oracle (remember that the number of oracle queries depends on $m$ but not on $|\mathcal{U}|$) each of which is of complexity $\mc{O}(|\mathcal{U}|)$, cf. Alg.~\ref{alg:oracle}. Since the initial number of hyperplanes is $2n$, the total number of iterations of the HPS algorithm is $\mc{O}(\log n)$. By noting that after iteration $k$ only $\alpha^{k}|\mathcal{U}|$ hyperplanes remain with unknown signs, with $0<\alpha<1$, it follows that the total complexity of the HPS algorithm is $\mc{O}(n)$. Parameter $\alpha$ depends on the implementation of the MDS algorithm, i.e. the number of queries to the oracle. For a more detailed argumentation the reader is kindly referred to \cite{Megiddo1984,Dyer1986,Megiddo1993}.
%
\end{proof}


Note that the linear-time bound for the HPS algorithm is valid only if the number of coupling constraints $m$ is fixed. Indeed, the constant of linearity that is "hidden" inside $\mc{O}(n)$ grows exponentially with $m$ \cite{Dyer1986}. Clearly, this limits practical applicability of the HPS algorithm to cases when $m$ is relatively small. The algorithm would still be linear with respect to $n$, but with a larger $m$ it would become slower and slower in practice, as demonstrated in Section~\ref{sec:illustrative_example} on a practical case study for $m=1$ and $m=2$.

{
	\begin{remark}
		Note that in case when $m=1$ the HPS algorithm boils down to the BPS algorithm. To better see the connection between the two algorithms one has to look at the difference in implementation of the oracle used by Alg.~\ref{alg:bps_algorithm} and Alg.~\ref{alg:hps_algorithm}. In Alg.~\ref{alg:bps_algorithm} ($m=1$) the oracle is implemented between lines 6--12, i.e. it simplifies to a check of whether the function value of \eqref{eq:glambda} at a breakpoint is equal to the right-hand side of the coupling constraint. This simpler implementation of oracle in Alg.~\ref{alg:bps_algorithm} is possible because of monotonicity of function \eqref{eq:glambda}. However, the oracle for the case $m=1$ could also be implemented like in Alg.~\ref{alg:oracle} (which is the implementation of oracle used by Alg.~\ref{alg:hps_algorithm} for a general case $m>1$), i.e. by checking the sub-gradients of \eqref{eq:phi_lam_bps} at a given breakpoint. Both oracle implementations would determine the same thing -- the position of the global optimizer $\lambda^\star$ relative to the given breakpoint. Hence the claim that the BPS algorithm is just a special case of the HPS algorithm.\vspace{-0.3cm}
	\end{remark}
}

\section{Illustrative example}
\label{sec:illustrative_example}

The efficiency of the proposed control algorithm will be demonstrated on an example from the domain of electrical distribution systems. The interested reader is referred to our previous work \cite{Spudic2013} where the proposed control algorithm was applied to the optimal control of wind farms, with $m=1$. In this paper, however, we consider the optimal coordination of a microgrid system, e.g. as in \cite{Novoselnik2015} but with two coupling constraint and with more subsystems.

A microgrid is a cluster of locally controllable distributed (renewable) generation sources, storages, and loads operating as a single controllable system \cite{Lasseter2002,Hatziargyriou2007}. As such, microgrid belongs to a class of systems of systems coupled by flows of energy. Microgrids where local subsystems provide and/or consume both electrical power and heat are considered (see Fig. \ref{fig:microgrid}). Microgrid concept is expected to enhance utilization and integration of distributed (and especially renewable) generation sources through the use of energy storage systems that enable the time-shift between production and consumption. Microgrids can be operated and managed independently from the power distribution grid and can economically optimize their internal power flows and the power exchange profile with the grid based on varying electricity prices, local energy needs, the states of the storage devices and renewable sources availability \cite{Parisio2014,Gulin2015}. Furthermore, they can also be used to stabilize voltage conditions in the overall power distribution grid, to shorten the energy path to consumers, and to minimize the $\text{CO}_2$ footprint of the distributed energy production \cite{Lasseter2011}. 

\begin{figure}[!t]
	\centering
	\includegraphics[width=0.43\textwidth]{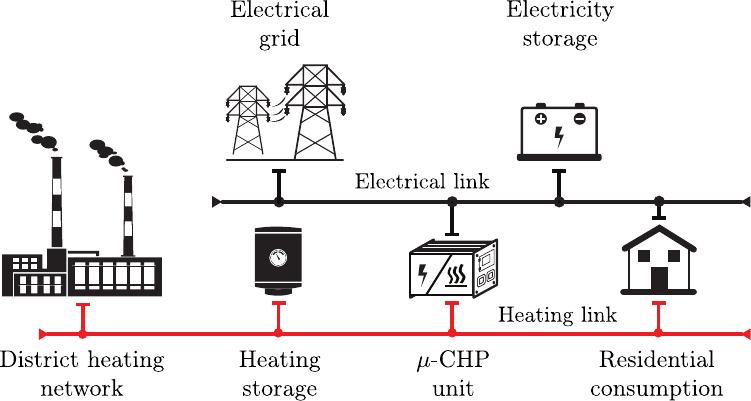}
	\caption{An illustration of a microgrid system comprising different generation, consumption, and storage devices that can provide and/or consume both electrical  and heating power.}
	\label{fig:microgrid}\vspace{-0.3cm}
\end{figure}

We consider a microgrid comprising $M$ controllable subsystems: (i) micro combined heat and power ($\mu$-CHP) units, and (ii) heating and electricity storage devices. There is also a certain number of uncontrollable heating and electricity consumers (see Fig. \ref{fig:microgrid}). All $\mu$-CHP units are grouped in set $\mc{G}$, all electricity storage units in set $\mc{S}_\tr{e}$, and all heating storage units in set $\mc{S}_\tr{h}$, such that $\vert \mc{G} \vert=\vert \mc{S}_\tr{e} \vert=\vert \mc{S}_\tr{h} \vert$ and $\vert \mc{G} \vert+\vert \mc{S}_\tr{e} \vert+\vert \mc{S}_\tr{h} \vert = M$. It is assumed that all subsystems share the same electrical/heating link. A prediction horizon of length $N=10$ is used.\vspace{-0.5cm}

\subsection{Mathematical models of individual subsystems}


\subsubsection{$\mu$-CHP units}

Micro CHP units are small-scale cogeneration units intended for homes or small commercial buildings. They can produce both electricity and heat from some primary energy source (e.g. natural gas). It is assumed that the $\mu$-CHP unit produces electricity primarily and heat is the by-product. The electrical efficiency of the $\mu$-CHP unit is denoted by $\eta_\tr{e}$ and its thermal efficiency by $\eta_\tr{h}$. For simplicity, no extra losses are considered, i.e. $\eta_\tr{h} = 1- \eta_\tr{e}$. The produced electrical and heating power are denoted by $p_{\tr{e}}$ and $p_{\tr{h}}$, respectively. It is evident that $p_{\tr{h}} = p_{\tr{e}} \eta_{\tr{h}}/\eta_{\tr{e}}$.

All $\mu$-CHP are modeled as second-order LTI systems with electrical efficiency $\etae$ chosen randomly from $\left[0.5,\;0.7\right]$ and system matrices\vspace{-0.1cm}
\begin{equation*}
	\begin{array}{l}
	A_i = \left[\begin{matrix}
	0.6+0.2 \zeta_{i} & -0.1-0.1\zeta_{i}\\
	1 & 0
	\end{matrix}\right], \ B_i = \left[\begin{matrix}
	\etae\\
	0
	\end{matrix}\right], \\[2ex]
	C_i = \left[\begin{matrix}
	1 & 0
	\end{matrix}\right], \ D_i = \left[0\right], \ i\in\mc{G},
	\end{array}\vspace{-0.1cm}
\end{equation*}
where $\zeta_{i}$ is a random number drawn uniformly from $\left[0,1\right]$. The output of $\mu$-CHP is $p_{\tr{e},i}$, i.e. the electrical power produced by the \ith $\mu$-CHP unit. The produced heat is simply $p_{\tr{h},i} = \frac{1-\etae}{\etae} p_{\tr{e},i}$. A power reference, which is to be tracked by the \ith $\mu$-CHP unit, is denoted by $p_{\tr{r},i}$.

The local objective of the \ith $\mu$-CHP unit reflects the desire to track the power reference but also penalizes the excessive use of input signal:
\begin{align*}
J_i = \textstyle\sum_{k=0}^{N-1} Q_i (C_i x_{k,i}-p_{\tr{r},i})^2 + R_i u_{k,i}^2,\quad i\in\mc{G},
\end{align*}
where $x_{k,i}$ and $u_{k,i}$ are the state and input, respectively, of the \ith $\mu$-CHP unit at time step $k$. Matrices $Q_i = 10 (1+4\zeta_i)$ and $R_i = 0.1 (1+\zeta_i)$ are used.

States and inputs are constrained as $\underline{x}_i\leq x_i \leq \overline{x}_i$ and $\underline{u}_i\leq u_i \leq \overline{u}_i$, respectively, where:
\begin{subequations}
	\begin{alignat*}{3}
	\underline{x}_i &= \left[0,0\right]\tran,\quad & \overline{x}_i &= (1+4\zeta_i) \left[20,20\right]\tran,&\quad &i\in\mc{G},\\
	\underline{u}_i &= 0, & \overline{u}_i &= \left[\overline{x}_i\right]_1/\etae,&\quad &i\in\mc{G}.
	\end{alignat*}
\end{subequations}

\subsubsection{Storage devices}
Both heating and electricity storage devices are modeled as integrators with the following system matrices (for simplicity we neglect losses):
\begin{equation*}
A_i = 1, \ B_i = \frac{-1}{20(1+4\zeta_i)}, \ C_i = 1, \ i\in\mc{S}_\tr{e}\cup\mc{S}_\tr{h},
\end{equation*}
where $\zeta_i$ is a random number drawn uniformly from $\left[0,1\right]$. State $x_i$ represents the amount of energy available in the storage device. Two reference signals are defined: (i) a locally defined reference for the state of charge $x_{\tr{r},i}$, and (ii) a reference for the power production/consumption $p_{\tr{r},i}$ that can be set by the microgrid coordinator.

The objective of the \ith storage device balances between tracking the locally defined reference for the state of charge and the externally defined reference for the power production/consumption:
\begin{align*}
J_i = \textstyle\sum_{k=0}^{N-1} Q_i (x_{k,i}-x_{\tr{r},i})^2 + R_i (u_{k,i} - p_{\tr{r},i})^2,   i\in\mc{S}_\tr{e}\cup\mc{S}_\tr{h},
\end{align*}
where $Q_i = (1+\zeta_i)I_2$ and $R_i = 10 (1+\zeta_i)$.

States and inputs are constrained as follows:
\begin{subequations}
	\begin{alignat*}{3}
	\underline{x}_i &= 0, & \overline{x}_i &= 1, &\quad &i\in\mc{S}_\tr{e}\cup\mc{S}_\tr{h},\\
	\overline{u}_i &= \frac{1}{5 \vert B_i \vert},\quad & \underline{u}_i &= - \overline{u}_i, &  \quad &i\in\mc{S}_\tr{e}\cup\mc{S}_\tr{h}.
	\end{alignat*}
\end{subequations}

\subsubsection{Consumers}
Power demand profiles are used to represent the consumers. These power demand profiles can be predicted, e.g. using historical data. The total electrical power demand is denoted by $\wt{p}_{\tr{e}}$ and the total heating power demand is denoted by $\wt{p}_{\tr{h}}$. Heat and electricity demand profiles during a typical day are shown in Fig. \ref{fig:demand}.\vspace{-0.2cm}

\begin{figure}[t]
	\centering
	\includegraphics[width=0.35\textwidth]{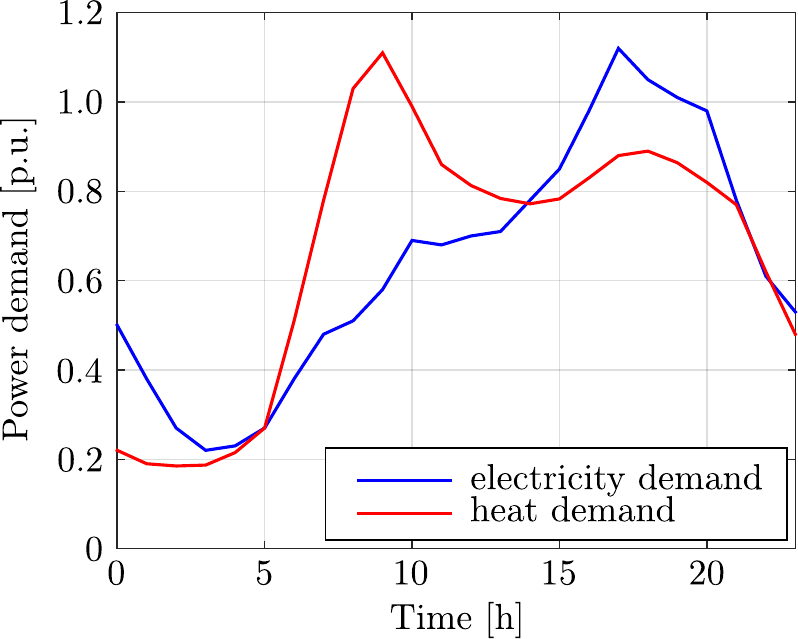}
	\caption{Typical heat and electricity demand profiles in p.u. during one day.}
	\label{fig:demand}\vspace{-0.5cm}
\end{figure}

\subsection{The control problem}

In the control problem that is considered, the microgrid coordinator needs to ensure that the total power output of the microgrid tracks the total power demand. This is achieved by distribution of the power references to individual subsystems, i.e. the task of the microgrid coordinator is to optimally coordinate individual subsystems to meet the common requirement while satisfying all constraints. It is assumed that microgrid operates in the grid-connected mode so the main electrical grid and district heating grid ensure the balance of both electrical and heating power (i.e. any excess generated power can be exported to the main grid and vice-versa). The global objective is the sum of locally defined objective functions for each subsystem. Such control problem can easily be formulated as an MPC problem. Furthermore, the considered control problem can be cast to form \eqref{eq:coordination_problem}. The local parameter contains the initial state of the subsystem and local references. The coordination parameter is the power output reference given to the individual subsystem. The coordination requirement states that the sum of individual power output references should be equal to the total power demand. In the simulations, two cases are considered:
\begin{enumerate}
	\item The first case is where only electrical power demand needs to be tracked. In this case there is only one coupling constraint so the global coordination problem \eqref{eq:evaluated_coordination_problem_formulation} can be solved using the BPS algorithm (see Subsection \ref{subsec:single_coupling_constraint}).
	
	\item The second case is where both the electrical power demand and the heating power demand are to be tracked. In this case we need to use the HPS algorithm (see Subsection \ref{subsec:multiple_coupling_constraints}) because there are now two coupling constraints.\vspace{-0.4cm}
\end{enumerate}

\subsection{Simulations and results}

In both cases, the computation time required to compute the solution using the proposed approach is compared to that of a classical on-line MPC implementation. The comparison is based on a number of simulations for different microgrid sizes (i.e. different numbers of $\mu$-CHP and storage devices in a microgrid). The simulations are done using Matlab 8.5.0 (R2015a) on a personal computer with Intel(R) Core(TM) i5 CPU at $3.4$GHz, with $8$ GB RAM, on Windows 10 operating system.

The classical approach (solving a QP \eqref{eq:coordination_problem} at every sampling instant) is tested using CPLEX (version 12.6) - a state of the art commercial QP solver that can exploit sparsity structures in QPs. YALMIP toolbox \cite{YALMIP} is used to formulate and solve the overall centralized optimization problem.

The proposed approach is implemented in Matlab. Off-line solutions are computed using the MPT toolbox \cite{MPT3}. The coordination algorithms, i.e. the BPS and the HPS algorithm, are implemented as Matlab functions, i.e. they are not implemented in C/C++ or a similar compiled programming language. For a median algorithm, however, we do use \texttt{std::nth\_element} from the C++ standard library, compiled as a mex-file for Matlab. This implementation of a partial sorting algorithm has a linear complexity on average, cf. \cite{CPP}. 

The simulations are done for a range of microgrid sizes of up to a thousand  subsystems. Each simulation is done for $168$ time steps where one time step equals to $1$ hour (which corresponds to $7$ days of total simulated time). The simulation results are shown in Fig. \ref{fig:histogram} and Fig. \ref{fig:experiment}. Figure \ref{fig:histogram} shows the histogram of computation times per instance of local evaluation in our approach (i.e. for evaluation on one subsystem) of $\wt{J}_i$ and $\wt{U}_i$. The computation time never exceeded 4.5 milliseconds and the average computation time was around 2.7 milliseconds.

Figure \ref{fig:experiment} depicts the comparison of the total computation times of a classical on-line MPC approach and our approach in both considered cases: the case of one coupling constraint is shown in Fig. \ref{fig:1eq} and the case of two coupling constraints in Fig. \ref{fig:2eq} . In both cases our approach clearly outperforms the classical MPC approach. In the case of one coupling constraint our approach is up to $100$ times faster than the classical MPC approach. In the case of two coupling constraints the speed-up is not that drastic but it is still present for microgrids comprising more than $100$ subsystems. We point out, however, that the results would probably be even better if our approach was fully implemented in highly efficient compiled programming language like C/C++ instead of an interpreted programming language like Matlab. 
The linear increase in computation time is clear for our approach in both Fig. \ref{fig:1eq} and Fig. \ref{fig:2eq}. Technically the complexity of the BPS and the HPS algorithms is linear with respect to the total number of variables in the transformed global coordination problem \eqref{eq:separable_QP}. However, recall that the number of variables is equal to the number of intervals in PPQW functions \eqref{eq:ppqw_fja} that are shared with the global coordinator by each subsystem. In our case study the average number of these intervals is similar for all subsystems (because our subsystems are similar) and averages to around 5 intervals per subsystem during all simulations. It follows that the total number of variables in \eqref{eq:separable_QP} was on average $5 M$ during our simulations and so the linear trend is evident with respect to the number of subsystems as well.

\begin{figure}[t]
	\centering
	\includegraphics[width=0.36\textwidth]{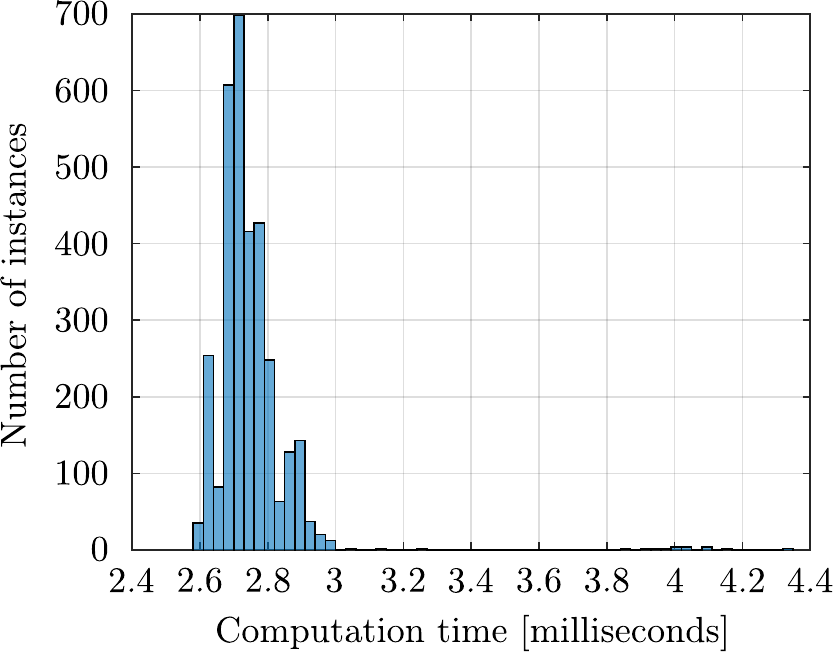}
	\caption{Histogram of the computation time for the local evaluation.}
	\label{fig:histogram}\vspace{-0.5cm}
\end{figure}

\begin{figure*}[!t]
	\centering
	\subfloat[The case with one coupling constraint.]
	{\includegraphics[height=2.2in]{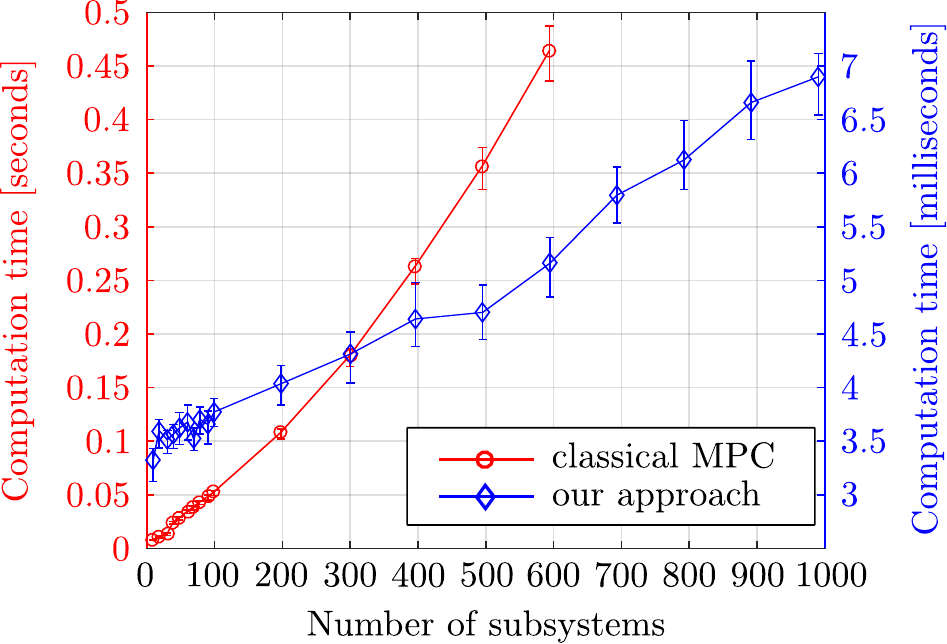}%
		\label{fig:1eq}}
	\hfil
	\subfloat[The case with two coupling constraints.]
	{\includegraphics[height=2.2in]{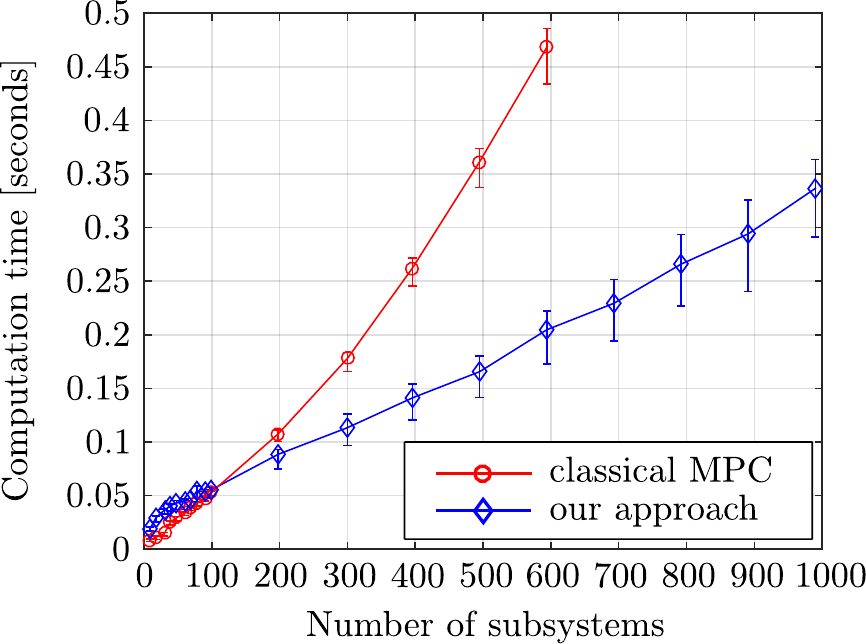}%
		\label{fig:2eq}}
	\caption{Comparison of computation times of a classical on-line MPC approach and our approach. Circles and diamonds denote mean computation time, while error bars denote the range of obtained computation times during simulations.}
	\label{fig:experiment}\vspace{-0.3cm}
\end{figure*}


\section{Conclusion}
\label{sec:conclusion}

The paper describes an efficient implementation of the MPC algorithm for the coordinated control of a large-scale System of Systems. The proposed method consists in distribution and parametrization of the overall control problem, which enables a significant part of the computational effort to be carried out off-line, by individual subsystems. The on-line computation -- finding the globally optimal solution -- is carried out by the coordinator. To achieve this the coordinator requires only a limited amount of information from subsystems that is sent once per sampling time, i.e. (unlike in classical distributed optimization techniques) there is no need for iterative communication between the coordinator and the subsystems. An algorithm is derived, for the coordinator's on-line computation, whose complexity grows linearly with the number of variables (for a fixed number of coupling constraints). There is no conservatism in the obtained solution in comparison to the classical centralized approach.

Using an example of a microgrid coordination controller design, it is shown that the proposed solution method can lead to drastic reductions in on-line computation times. In case of a single coupling constraint the on-line computation time for the proposed approach is up to two orders of magnitude smaller than the on-line computation time when the classical MPC controller implementation is used. Although the obtained speed-up is less drastic in the case of multiple coupling constraints, it still illustrates that the optimal control of large--scale systems at small sampling times is achievable.


%

\appendices
\section*{Appendix}

{
	The purpose of this Appendix is to show that problem \eqref{eq:evaluated_coordination_problem_formulation} can be reformulated as \eqref{eq:separable_QP}. For this purpose we introduce the following Lemma~\ref{lem:monotone} and Theorem~\ref{th:CPWQmonotoneOpt} which are then used to demonstrate the reformulation of \eqref{eq:evaluated_coordination_problem_formulation} to \eqref{eq:separable_QP}.
}

\label{appendix}

\begin{lemma}
\label{lem:monotone}
Let $\varphi:\RR\rightarrow\RR$ be a non-affine, convex function on an interval $[x_1, x_2]\subset \RR$, with $x_1<x_2$. Then
\begin{equation}
\label{eq:monotonediff}
\begin{array}{r}
\varphi(x_1+\xi)+\varphi(x_2-\xi) < \varphi(x_1)+\varphi(x_2), \\[1.5ex]
\forall \xi\in(0, x_2-x_1).
\end{array}
\end{equation}
\end{lemma}
\begin{proof}
Note that $x_1<x_1+\xi<x_2$ and $x_1<x_2-\xi<x_2$ for all $\xi\in(0, x_2-x_1)$. Since $\varphi$ is a non-affine, convex function on $[x_1, x_2]$ the following inequalities hold
\begin{equation*}
\textstyle\varphi(x_1+\xi)< \varphi(x_1)+ \frac{\varphi(x_2)-\varphi(x_1)}{x_2-x_1}\xi,
\end{equation*}
\begin{equation*}
\textstyle\varphi(x_2-\xi)< \varphi(x_1)+ \frac{\varphi(x_2)-\varphi(x_1)}{x_2-x_1}(x_2-\xi-x_1),
\end{equation*}
which can be easily combined to obtain \eqref{eq:monotonediff}.
\end{proof}

\begin{theorem}
\label{th:CPWQmonotoneOpt}
Consider the following optimization problem
\begin{equation}
\label{eq:equivalentProblem}
\begin{array}{cl}
\min\limits_{w, y_1,\ldots,y_N} & \gamma(w) +\varphi(Z_0)+\sum\limits_{r=1}^{N} [\varphi(y_r)-\varphi(Z_{r-1})]\\[2ex]
\mathrm{s.t.} & Z_{r-1} \leq y_r \leq Z_{r}, \ r=1,\ldots,N,\\
& \bar{A}\left[Z_0+\sum\limits_{r=1}^{N}[y_r-Z_{r-1}] \right] + \bar{B}w\leq \bar{C},
\end{array}
\end{equation}
where $\gamma:\RR^{n_{\mathrm{w}}}\rightarrow\RR$ is a convex piecewise quadratic function, $\bar{A}\in\RR^{m}$, $\bar{B}\in\RR^{m\times n_{\mathrm{w}}}$, $\bar{C}\in\RR^{m}$, and $\varphi:\RR \rightarrow\RR$ is a convex piecewise quadratic function on an interval $[Z_{0},Z_{N}]\subset\RR$,
\begin{equation}
\begin{array}{r}
\varphi(z)= \dfrac{1}{2} h_r z^2 + f_r z + g_r \ \mathrm{if}~ z\in[Z_{r-1}, Z_{r}],\\[1.5ex]
 r\in\{1,\ldots,N\},
\end{array}
\end{equation}
with $Z_0<Z_1<\ldots<Z_{N}$, $N$ is the number of subintervals (regions), and it is assumed, without loss of generality\footnote{Otherwise one could simply reduce $N$ by merging two neighboring subintervals for which $\varphi$ has the same affine expression.}, that coefficients $h_r, f_r, g_r\in\RR$ are such that 
\begin{equation}
h_r>0 ~\vee~ [f_r \; g_r] \neq [f_{r+1} \; g_{r+1}], \ \forall r\in\{1,\ldots,N-1\}.
\end{equation}
Let $w^*$, $y^*_r$, $r=1,\ldots,N$, be an optimizer of \eqref{eq:equivalentProblem}, then
\begin{itemize}
\item[i)] $\forall r\in\{2,\ldots,N\}$,
\begin{equation}
\label{eq:stickRight}
\mathrm{if}~ y^*_r>Z_{r-1} ~\mathrm{then}~ y^*_{s}=Z_{s}, \ s=1,\ldots,r-1,
\end{equation}
\item[ii)] $\forall r\in\{1,\ldots,N-1\}$,
\begin{equation}
\label{eq:stickLeft}
\mathrm{if}~ y^*_r=Z_{r-1} ~\mathrm{then}~ y^*_{s}=Z_{s-1}, \ s=r+1,\ldots,N,
\end{equation}
\item[iii)] $w^*$ and $z^*\in\RR$, with 
\begin{equation}
\label{eq:optimizer}
\textstyle z^*:=Z_0+\sum_{r=1}^{N}[y^*_r-Z_{r-1}],
\end{equation}
is an optimizer of the following problem
\begin{equation}
\label{eq:initialProblem}
\begin{array}{cl}
\min\limits_{w, z} & \gamma(w) +\varphi(z)\\[2ex]
\mathrm{s.t.} & Z_{0} \leq z \leq Z_{N},\\
& \bar{A}z + \bar{B}w\leq \bar{C}.
\end{array}
\end{equation}
\end{itemize}
\end{theorem}
\begin{proof}

\begin{itemize}
\item[i)] It is enough to prove \eqref{eq:stickRight} for $s=r-1$, since the other cases ($s<r-1$) follow analogously. Assume the opposite, $\exists r\in\{2,\ldots,N\}$ such that $y^*_r>Z_{r-1}$ and $y^*_{r-1}<Z_{r-1}$. Choose any $\delta\in\left(0,\min\{y^*_r-Z_{r-1}, \ Z_{r-1}-y^*_{r-1}\}\right)$. From Lemma~\ref{lem:monotone} (by using $x_1=y^*_{r-1}$, $x_2=y^*_{r}$, $\xi=\delta$) it follows that $\varphi(y^*_{r-1}+\delta)+\varphi(y^*_{r}-\delta)< \varphi(y^*_{r-1})+\varphi(y^*_{r})$.
Therefore, one can construct a feasible point for \eqref{eq:equivalentProblem}: $w=w^*$, $y_{r-1}=y^*_{r-1}+\delta$, $y_{r}=y^*_{r}-\delta$, and $y_{i}=y^*_{i}$, $\forall i\in\{1,\ldots,r-2,r+1,\ldots,N\}$, that gives a smaller value of the objective function in \eqref{eq:equivalentProblem} than the optimizer does -- a contradiction. Therefore, it is proven $y_{s}^\star \ge Z_{s}$, but from bounds on $y_{s}$ it follows $y_{s}^\star = Z_{s}$.
\item[ii)] Similarly as in i). Let $y^*_r=Z_{r-1}$ and $y^*_{r+1}>Z_{r}$. From Lemma~\ref{lem:monotone}, for any $\delta\in\left(0,\min\{y^*_{r+1}-Z_{r}, \ Z_{r}-Z_{r-1}\}\right)$ follows $\varphi(y^*_{r}+\delta)+\varphi(y^*_{r+1}-\delta)< \varphi(y^*_{r})+\varphi(y^*_{r+1})$. Hence, a feasible point for \eqref{eq:equivalentProblem} exists: $w=w^*$, $y_{r}=y^*_{r}+\delta$, $y_{r+1}=y^*_{r+1}-\delta$, and $y_{i}=y^*_{i}$, $\forall i\in\{1,\ldots,r-1,r+2,\ldots,N\}$, that gives smaller value of the objective function in \eqref{eq:equivalentProblem} than the optimizer -- a contradiction. Therefore, it is proven $y_{s}^\star \le Z_{s-1}$, but from bounds on $y_{s}$ it follows $y_{s}^\star = Z_{s-1}$.
%
%
\item[iii)] Note that any $z\in[Z_{0}, \ Z_{N}]$ can be written as
\begin{equation}
\label{eq:connection}
\textstyle z=Z_0+\sum_{r=1}^{N}[y_r-Z_{r-1}],
\end{equation}
with appropriate (not necessarily unique) choice of values for new variables
\begin{equation}
\label{eq:yrconstr}
y_r\in\RR, \quad Z_{r-1}\leq y_r \leq Z_{r}, \quad r=1,\ldots,N.
\end{equation}
It can be easily verified that the following, particular choice of $y_r=y_r(z)$:
\begin{equation}
\label{eq:z2yr}
y_r=\left\{
\begin{array}{cl}
Z_{r-1} & \mathrm{if}~ z<Z_{r-1},\\
z & \mathrm{if}~ Z_{r-1}\leq z \leq Z_{r},\\
Z_{r} & \mathrm{if}~ z>Z_{r},
\end{array}
\right. \ r=1,\ldots,N,\vspace{-0.2cm}
\end{equation}
satisfies \eqref{eq:connection}-\eqref{eq:yrconstr} and also guarantees that
\begin{equation}
\label{eq:eqvCost}
\textstyle\varphi(z)=\varphi(Z_0)+\sum_{r=1}^{N}[\varphi(y_r)-\varphi(Z_{r-1})].
\end{equation}
Consequently, the problem \eqref{eq:equivalentProblem} is really a relaxation of the problem \eqref{eq:initialProblem} -- with omitted constraints \eqref{eq:z2yr}. However, from i) and ii) it follows that an optimizer of \eqref{eq:equivalentProblem} -- variables $y^*_r$, $r=1,\ldots,N$ -- behaves as if defined with \eqref{eq:z2yr}, with $y_r=y^*_r$ and $z=z^*$, where $z^*$ is given by \eqref{eq:optimizer}. Since this implies that \eqref{eq:eqvCost} holds (with $y_r=y^*_r$ and $z=z^*$), it follows that both problems achieve the same optimal cost. Therefore, $z^*$ is an optimal solution to \eqref{eq:initialProblem}.
\end{itemize}

\end{proof}

Note that problem \eqref{eq:evaluated_coordination_problem_formulation} can be written as \eqref{eq:initialProblem}. We have simply singled out one scalar convex piecewise quadratic function $\varphi(z)$ while the sum of the remaining functions from \eqref{eq:evaluated_coordination_problem_formulation} is replaced by a single convex piecewise quadratic function $\gamma(w)$ in \eqref{eq:initialProblem} to simplify the notation. Theorem~\ref{th:CPWQmonotoneOpt} proves that \eqref{eq:initialProblem} can be solved by solving its relaxation \eqref{eq:equivalentProblem}. Since Theorem~\ref{th:CPWQmonotoneOpt} holds for an arbitrary convex piecewise quadratic function $\gamma:\RR^{n_{\mathrm{w}}}\rightarrow\RR$, its results can be applied (in succession) to $\gamma$ that is the sum of scalar convex piecewise quadratic functions (which is the case in problem \eqref{eq:evaluated_coordination_problem_formulation}), i.e. the same relaxation that is mapped out in Theorem~\ref{th:CPWQmonotoneOpt} for a single scalar piecewise quadratic function $\varphi(z)$ can be repeated for the remaining functions hidden in $\gamma(w)$. By ignoring constant parts of $\wt{J}_i$, problem \eqref{eq:evaluated_coordination_problem_formulation} transforms to:
\begin{subequations}
	\label{eq:transformed_coordination_problem}
	\begin{alignat}{2}
	& \underset{\substack{y_{1,1},\ldots,y_{M,N_M}}}{\text{min}}
	& & \; \textstyle\sum\limits_{i=1}^{M}\sum\limits_{r=1}^{N_i}\frac{1}{2}h_{i,r} y_{i,r}^2+f_{i,r}y_{i,r}, \\
	&\hspace{0.76cm} \text{s.t.} 
	& & \; \begin{aligned}
	I_{i,r-1}\le y_{i,r}\le I_{i,r},\;\; r&=1,\ldots,N_i,\\
	i&=1,\ldots,M,
	\end{aligned}\\
	& & & \; \begin{aligned}
	\textstyle\sum\limits_{i=1}^{M} a_{i,j} \left[I_{i,0} + \sum\limits_{r=1}^{N_i} [y_{i,r} - I_{i,r-1}] \right] = b_j, \\
	\quad j=1,\ldots,m.
	\end{aligned}
	\end{alignat}
\end{subequations}
Finally, with a substitution of variables $\theta_{i,r}=y_{i,r}-I_{i,r-1}$ and by ignoring the constant part of the cost again, \eqref{eq:transformed_coordination_problem} becomes the following separable quadratic program with box constraints and coupling equality constraints:
\begin{equation}
\label{eq:oneshot_coordination_problem_formulation}
\begin{array}{cl}
\underset{\theta_{1,1},\ldots,\theta_{M,N_M}}{\text{min}}
&\hspace{-0.3cm} {\textstyle\sum\limits_{i=1}^{M}}{\textstyle\sum\limits_{r=1}^{\nr}} \frac{1}{2} h_{i,r} \theta_{i,r}^2 + (h_{i,r} I_{i,r-1} + f_{i,r}) \theta_{i,r}, \\[3ex]
\mathrm{s.t.} 
&\hspace{-0.3cm} 0 \leq \theta_{i,r} \leq I_{i,r} - I_{i,r-1}, \ r=1,\ldots,\nr, \\
&\hspace{-0.3cm} \phantom{0 \leq \theta_{i,r} \leq I_{i,r} - I_{i,r-1},} \ i=1,\ldots,M, \\
&\hspace{-0.3cm} {\textstyle\sum\limits_{i=1}^{M}} a_{i,j} \Big[I_{i,0} + {\sum\limits_{r=1}^{\nr}} \theta_{i,r} \Big] = b_j, j=1,\ldots,m.
\end{array}
\end{equation}
Therefore, from Theorem~\ref{th:CPWQmonotoneOpt} it follows that the optimizer, $\Theta_{i}^{\star}$, $i=1,\ldots,M$, for problem \eqref{eq:evaluated_coordination_problem_formulation} can be calculated as
\begin{equation}
\label{eq:theta_star}
\Theta_{i}^{\star} = I_{i,0} + \textstyle\sum\limits_{r=1}^{\nr} \theta_{i,r}^{\star}, \quad i=1,\ldots,M.\vspace{-0.3cm}
\end{equation}

\ifCLASSOPTIONcaptionsoff
  \newpage
\fi



\bibliographystyle{IEEEtran}
\bibliography{IEEEabrv,bibliography}\vspace{-1cm}
%
%
%

%

\begin{IEEEbiography}[{\includegraphics[width=1in,height=1.25in,clip,keepaspectratio]{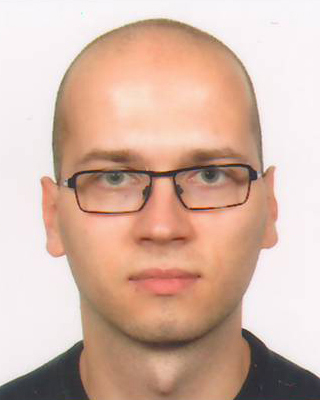}}]{Branimir Novoselnik} 
	
	received his M.Sc. and Ph.D. degrees, both in Electrical Engineering, from the University of Zagreb, Faculty of Electrical Engineering and Computing (UNIZG-FER), Croatia, in 2013 and 2018, respectively. Currently he is Research and Teaching Assistant at the Department of Control and Computer Engineering, UNIZG-FER, Croatia.
	
	His research interests include optimal control, mathematical programming, and model predictive control.\vspace{-1cm}
\end{IEEEbiography}

\begin{IEEEbiography}[{\includegraphics[width=1in,height=1.25in,clip,keepaspectratio]{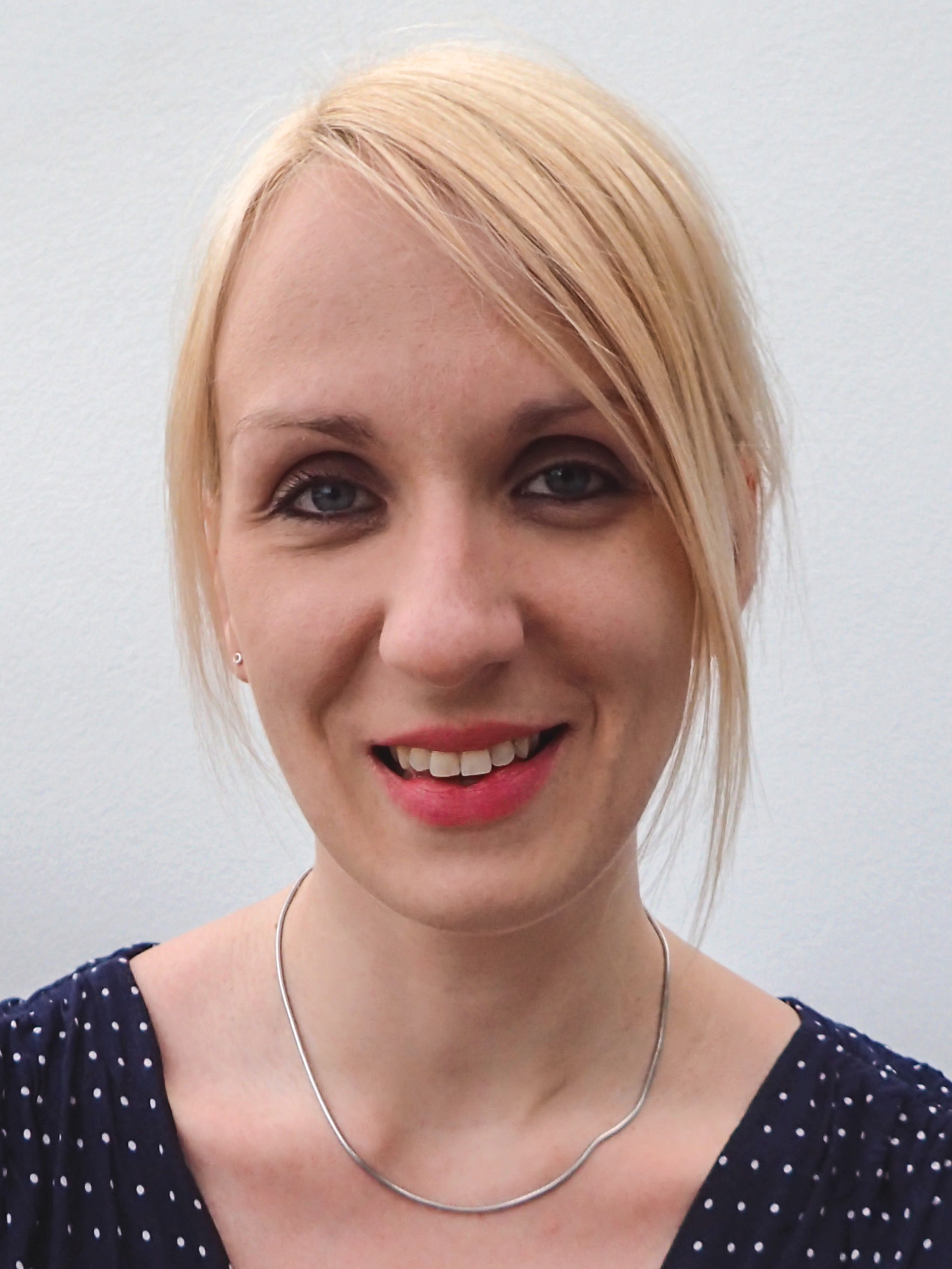}}]{Vedrana Spudi\'{c}}
received her M.Sc. and Ph.D. degrees in electrical engineering from University of Zagreb, in 2008 and 2012, respectively. Since 2014 she works at ABB Corporate Research Center in Baden-Dättwil, Switzerland, where she is currently leading the Power Conversion Systems research group.

Her research interests include mathematical optimization and optimal control applied to fast real time systems. Recently she is focusing on advanced control for power electronics applications.\vspace{-1cm}
\end{IEEEbiography}


\begin{IEEEbiography}[{\includegraphics[width=1in,height=1.25in,clip,keepaspectratio]{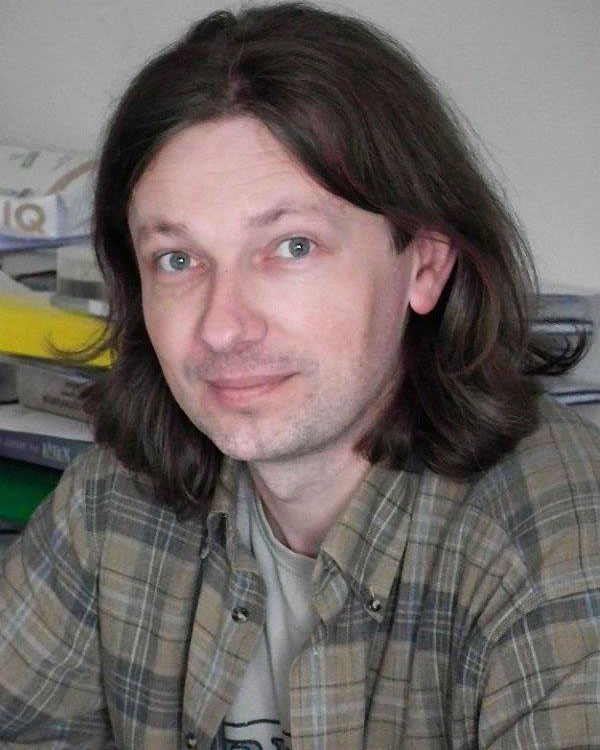}}]{Mato Baoti\'{c}}
	
	received the B.Sc. and M.Sc. degrees, both in Electrical Engineering, from the University of Zagreb, Faculty of Electrical Engineering and Computing (UNIZG-FER), Croatia, in 1997 and 2000, respectively, and the Ph.D. from the ETH Zurich, Switzerland, in 2005. Currently, he is Professor with the Department of Control and Computer Engineering, UNIZG-FER, Zagreb, Croatia. 
	
	His research interests include mathematical programming, hybrid systems, optimal control, and model predictive control.
\end{IEEEbiography}




\end{document}